\documentclass[english]{article}
\usepackage[T1]{fontenc}
\usepackage[latin9]{inputenc}
\usepackage{geometry}
\usepackage{babel}
\usepackage{amsmath}
\usepackage{amsthm}
\usepackage{amssymb}
\usepackage{enumitem}
\usepackage[numbers]{natbib}
\usepackage[all]{xy}
\usepackage[unicode=true,pdfusetitle,
 bookmarks=true,bookmarksnumbered=false,bookmarksopen=false,
 breaklinks=false,pdfborder={0 0 1},backref=false,colorlinks=false]
 {hyperref}
\usepackage{macros}
\usepackage{eucal}

\makeatletter
\theoremstyle{plain}
\newtheorem{thm}{\protect\theoremname}[section]
  \theoremstyle{definition}
  \newtheorem{defn}[thm]{\protect\definitionname}
  \newtheorem{prob}[thm]{Problem}
  \newtheorem{cons}[thm]{Construction}
    
      \newtheorem{rem}[thm]{\protect\remarkname}
  \newtheorem{warn}[thm]{Warning}
  \newtheorem{scenario}[thm]{Scenario}
  \theoremstyle{remark}

  \theoremstyle{plain}
  \newtheorem{prop}[thm]{\protect\propositionname}
  \theoremstyle{plain}
  \newtheorem{lem}[thm]{\protect\lemmaname}
   \newtheorem{schol}[thm]{Scholium}
  \theoremstyle{plain}
  \newtheorem{cor}[thm]{\protect\corollaryname}
  
  \numberwithin{equation}{section}

\usepackage{xy}
\usepackage{mathrsfs}
\usepackage{hyperref}
\usepackage{color}
\hypersetup{colorlinks=true,citecolor=blue}

\makeatother

  \providecommand{\corollaryname}{Corollary}
  \providecommand{\definitionname}{Definition}
  \providecommand{\lemmaname}{Lemma}
  \providecommand{\propositionname}{Proposition}
  \providecommand{\remarkname}{Remark}
\providecommand{\theoremname}{Theorem}

\begin{document}

\title{An enhanced six-functor formalism for diamonds and v-stacks
}

\author{Daniel Gulotta, David Hansen and Jared Weinstein}

\maketitle

\abstract{This article extends the six functor formalism for diamonds \cite{ECoD} to a very general class of stacky morphisms between v-stacks, using $\infty$-categorical techniques developed by Liu-Zheng \cite{LZ}.}

\tableofcontents

%
%

\section{Introduction}

\subsection{Context and motivation}

Grothendieck's six functors refer to the operations $f^*,Rf_*,Rf_!,Rf^!$, $\otimes$, and internal Hom in a suitable derived category of sheaves on some category of geometric objects.  In the original context, the category was schemes, but for some applications one may wish to extend the formalism to Artin stacks.  Considerable problems arise in defining $Rf_!$ when $f$ is a stacky morphism, owing to issues of homotopy coherence.  These problems were overcome in the work of Lu-Zheng \cite{LZ}, which applies $\infty$-categorical techniques to extend the full formalism to Artin stacks.

This article solves the analogous problem in the setting of perfectoid spaces and diamonds.  Fix a prime $p$ and a ring $\Lambda$ killed by some integer prime
to $p$. In a very important paper \cite{ECoD}, Scholze laid foundations
for the six functor formalism in the \'etale cohomology of diamonds. In particular,
he defined a triangulated category of \'etale sheaves $D_{\mathrm{\acute{e}t}}(X,\Lambda)$
for any small v-stack $X$, agreeing with the left-completion of $D(X_{\mathrm{\acute{e}t}},\Lambda)$
when e.g. $X$ is an analytic adic space, and equipped with the usual
$\otimes$ and $R\mathscr{H}\mathrm{om}$ operations. Moreover, for
any morphism $f:X\to Y$ he constructed a pair of adjoint functors
$f^{\ast}$ and $Rf_{\ast}$, and for morphisms $f$ which are \emph{compactifiable,
representable in locally spatial diamonds, and locally of finite dim.trg
}he constructed the exceptional pushforward and pullback functors
$Rf_{!}$ and $Rf^{!}$ with all their expected properties: composability,
base change, the projection formula, etc. Finally, he showed that there
is a reasonable notion of ``cohomologically smooth'' morphisms, for which $Rf^{!}$
is an invertible twist of $f^{\ast}$.

For applications it is important to extend Scholze's construction of $Rf_!$ and $Rf^!$ to some $f$ which are stacky (i.e., not representable).\footnote{A simple example is the structure map $f\colon [\ast/\underline{G(\mathbf{Q}_p)}]\to \ast$, where $G/\mathbf{Q}_p$ is a linear algebraic group.  Here one expects $Rf_!$ to be group homology (whereas $Rf_{\ast}$ is group cohomology). See \cite[Example 4.2.7]{HKW}.}  In particular, the results on the Kottwitz conjecture proved in \cite{HKW}
make crucial use of the $!$-functor formalism with its expected
properties for certain stacky maps of Artin v-stacks.\footnote{In fact, despite numerous attempts, DH and JW were never able to arrange
the arguments in \cite{HKW} in such a way as to avoid the use of a
stacky $!$-functor formalism.} In the present paper, we extend the !-functor formalism of \cite{ECoD}
to a specific class of stacky maps, as required by \cite{HKW}. We expect this
extended formalism to have many other applications.

\subsection{Main results}

In this section we give a precise statement of our main result.  We first define the category of geometric objects to which our formalism applies.

\begin{defn}\label{decent}
A small v-stack $X$ is \emph{decent} if it satisfies the following
two conditions:
\end{defn}

\begin{enumerate}
\item[(1)] For any locally separated locally spatial diamond $T$ with a map
$T\to X\times X$, the pullback $T\times_{X\times X,\Delta}X$ is
a locally separated locally spatial diamond.
\item[(2)] There exists a locally separated locally spatial diamond $U$ together
with a surjective map $f:U\to X$ which is representable in locally
spatial diamonds and which locally on $U$ is separated and cohomologically
smooth.
\end{enumerate}
Any choice of $f:U\to X$ as in (2) will be called a \emph{chart }of
$X$.
\begin{rem}\label{decentremarks}
i. In (1) it is enough to quantify over all separated spatial $T$. Moreover, (1) implies
that for any locally separated locally spatial diamonds $T,S$ with
maps $T\to X\leftarrow S$, the fiber product $T\times_{X}S$ is a
locally separated locally spatial diamond. Indeed, this fiber product
can be written as $(T\times S)\times_{X\times X,\Delta}X$, so the
claim follows by observing that $T\times S$ is a locally separated
locally spatial diamond.

ii. In (2), we can assume without loss of generality that $U$ is a separated locally spatial
diamond and that $f$ is separated and cohomologically smooth, for instance
by replacing a given chart $f$ by the composition $\coprod U_{i}\to U\overset{f}{\to}X$
for some open cover $U=\bigcup U_{i}$ by separated locally spatial
diamonds such that all restrictions $f|_{U_{i}}$ are separated and cohomologically smooth. If these conditions hold, we say $f:U \to X$ is a \emph{clean chart}.
\end{rem}

We will see (Proposition \ref{decentartin}) that decent v-stacks are Artin v-stacks in the sense of
\cite[Definition IV.1.1]{FS}. Generally, the condition of being decent is a very mild restriction. For instance, all locally separated locally spatial
diamonds are decent; in particular, for any analytic adic space $X$ over $\mathrm{Spa}\,\mathbf{Z}_{p}$,
the associated diamond $X^{\lozenge}$ is decent.
More generally, if $X$ is any locally separated v-sheaf such that there
exists a separated locally spatial diamond $U$ with a surjective cohomologically
smooth map $U\to X$, then $X$ is decent. This applies, for instance,
to $X=\Spd\mathbf{Z}_{p}$ and $X=\Div^1$.  Neither of these is a diamond, but in both cases, $U=X\times \Spd \mathbf{F}_p\laurentseries{t^{1/p^\infty}}$ is representable by an analytic adic space over $\mathbf{Z}_p$ (and is thus a locally spatial diamond), and $U\to X$ is surjective and cohomologically smooth.

Continuing this line of thought, one
can check that all Artin v-stacks appearing in \cite{FS} and \cite{HKW}
are decent, and in fact it takes some work to find an example of an Artin v-stack which is not decent.
\begin{defn}\label{fine}
i. A morphism $f:X\to Y$ between decent v-stacks is \emph{fine} if there exists a commutative diagram
\[
\xymatrix{W\ar[d]_{b}\ar[r]^{g} & V\ar[d]^{a}\\
X\ar[r]_{f} & Y
}
\]
where the vertical maps are charts and $g$ is locally on $W$ compactifiable
of finite dim.trg.

ii. A morphism $f:X\to Y$ between decent v-stacks is \emph{$\ell$-cohomologically
smooth} if there exists a commutative diagram
\begin{equation}
\label{EqCommDiagFine}
\xymatrix{W\ar[d]_{b}\ar[r]^{g} & V\ar[d]^{a}\\
X\ar[r]_{f} & Y
}
\end{equation}
where the vertical maps are charts and $g$ is locally on $W$ compactifiable
of finite dim.trg. and $\ell$-cohomologically smooth in the sense
of \cite{ECoD}.
\end{defn}

We will see that these classes of morphisms are quite reasonable:
they are stable under composition and base change, membership can
be tested smooth-locally on the source and target, etc. We note that the name ``fine'' is chosen to hint at the idea of being ``locally of finite type'', and also because these morphisms are ``good enough'' for all practical purposes.

The main result of this paper is the following theorem.
\begin{thm}\label{mainthm}
Let $\Lambda$ be a ring killed by some integer $n$ prime to $p$. If $f:X\to Y$ is any fine map of decent v-stacks, there is natural functor $Rf_{!}:D_{\mathrm{\acute{e}t}}(X,\Lambda)\to D_{\mathrm{\acute{e}t}}(Y,\Lambda)$ satisfying the following properties:
\begin{enumerate}
\item[(1)] When $f$ is separated and representable in locally spatial diamonds, $Rf_!$ coincides with the functor constructed in \cite{ECoD}.
\item[(2)] There is a natural isomorphism of functors $R(f\circ g)_!\isom Rf_!\circ Rg_!$ whenever fine morphisms $f$ and $g$ are composable.  More precisely, the assignments $X\mapsto D_{\et}(X,\Lambda)$ and $f\mapsto Rf_!$ upgrade to a pseudo-functor from the 2-category of decent v-stacks with fine morphisms, to the 2-category of triangulated categories.
\item[(3)] The projection formula:  there is a natural isomorphism $Rf_!(A\otimes f^*B) \cong Rf_!A \otimes B$ for $A\in D_{\et}(X,\Lambda)$ and $B\in D_{\et}(Y,\Lambda)$.  
\item[(4)] Proper base change:  For a cartesian square of decent v-stacks
\[
\xymatrix{
X' \ar[r]^{\tilde{g}} \ar[d]_{f'} & X \ar[d]^f \\
Y'\ar[r]_{g}& Y 
}
\]
with $f$ fine, there is a proper base change isomorphism $g^*Rf_! \isom Rf'_!\tilde{g}^*$.
\item[(5)] $Rf_!$ admits a right adjoint $Rf^!$.  
\end{enumerate}
\end{thm}

A heuristic for the construction of $Rf_{!}$ involves \emph{cohomological codescent.
}  To explain the idea, let $f:X\to Y$ be a fine map between decent
v-stacks, and let $g:U\to X$ be a chart. Set \[g_{n}:U_{n}=\underbrace{U\times_{X}U\cdots\times_{X}U}_{n+1}\to X,\]
so $U_{n}$ is a locally separated locally spatial diamond, and $g_{n}$
and $f\circ g_{n}$ are all fine and 0-truncated. Suppose that we
already have access to $Rh_{!}$ for fine 0-truncated maps $h$. We would
then like to define $Rf_{!}$ by the formula
\begin{equation}
\label{EqRfHeuristic}
Rf_{!}A=\mathrm{colim}_{n \in\boldsymbol{\Delta}}R(f\circ g_{n})_{!}Rg_{n}^{!}A.
\end{equation}
where $\boldsymbol{\Delta}$ is the simplex category.
One can deduce formally from \eqref{EqRfHeuristic} that proper base change and the projection formula hold for $Rf_!$.  For instance, to establish the projection formula:
\begin{alignat*}{2}
Rf_!(A\otimes f^*B)
\isom& 
\colim_{n\in\boldsymbol{\Delta}} R(f\circ g_n)_! Rg_n^! (A\otimes f^*B) && \\
\isom& \colim_{n\in\boldsymbol{\Delta}} R(f\circ g_n)_! (Rg_n^!A \otimes Rg_n^*f^*B) 
&& \text{\;\;(smoothness of $g_n$)} \\
\isom& \colim_{n\in\boldsymbol{\Delta}} R(f\circ g_n)_!Rg_n^!A \otimes B &&\text{\;\;(projection formula)}\\
\isom& Rf_!A \otimes B &&\text{\;\;($\otimes$ is closed)}
\end{alignat*}

The heuristic \eqref{EqRfHeuristic} follows formally from the expectation that the map
\begin{equation}
\label{EqCohoCodenscent}
\mathrm{colim}_{n\in\boldsymbol{\Delta}}Rg_{n}{}_{!}Rg_{n}^{!}A\to A
\end{equation}
should be an isomorphism, i.e. that ``cohomological codescent for
surjective smooth maps'' should hold.  To see the implication, apply $Rf_{!}$ to the equivalence in 
\eqref{EqCohoCodenscent} to obtain
\begin{align*}
Rf_{!}A & \simeq Rf_{!}\mathrm{colim}_{n\in\boldsymbol{\Delta}}Rg_{n}{}_{!}Rg_{n}^{!}A\\
 & \simeq\mathrm{colim}_{n\in\boldsymbol{\Delta}}Rf_{!}Rg_{n}{}_{!}Rg_{n}^{!}A\\
 & \simeq\mathrm{colim}_{n\in\boldsymbol{\Delta}}R(f\circ g_{n})_{!}Rg_{n}^{!}A,
\end{align*}
using the fact that $Rf_{!}$ is a left adjoint to pass
it across the colimit.

There are two main difficulties with making sense of this heuristic.
First and foremost, the colimit in \eqref{EqRfHeuristic} is not guaranteed to exist, let alone behave naturally as in Theorem \ref{mainthm}(2), since colimits in triangulated categories are generally ill-behaved.    

The other main difficulty is that even if $g$ as above was chosen to
be separated (which is always possible), $g_{n}$ and $f\circ g_{n}$
will only be locally separated in general, so the formalism in \cite{ECoD}
does not apply to construct $R(f\circ g_{n})_{!}$ and $Rg_{n}^{!}$.  

The remedy for the first difficulty is to upgrade the functors $R(f\circ g_{n})_{!}$ and $Rg_{n}^{!}$ to an $\infty$-categorical setting where colimits are well-behaved. This will be accomplished using the machinery of {\em enhanced operation maps} developed
in \cite{LZ, LZ2} in the context of schemes. We will review this machinery in a more detail in \ref{SectionEOMaps}, but for the moment we present a brief outline of how enhanced operations were used to produce the six functor formalism for Artin stacks.

Let $\CC$ be a category of geometric objects (e.g., schemes or Artin stacks) equipped with some marked classes of morphisms $\E_1,\dots,\E_n$, and a subset $I\subset \set{1,\dots,n}$.  An enhanced operation map for $(\CC,\E_1,\dots,\E_n,I)$ associates to each object $X\in \CC$ a symmetric monoidal $\infty$-category $\D(X)$, and it also associates to each morphism $X\to Y$ in $\E_i$ a functor $\D(X)\to \D(Y)$ (if $i\not\in I$) or a functor $\D(Y)\to \D(X)$ (if $i\in I$).  This must be done in such a way that is (a) compatible with cartesian diagrams built out of morphisms in the $\E_i$ and (b) compatible with the symmetric monoidal structure on $\D(X)$.

The main theorem of \cite{LZ} asserts the existence of an enhanced operation map for (Artin stacks, locally finite type morphisms, all morphisms, $\set{2}$), which encodes the operations $f^*$ and $Rf_!$.  The compatibilities inherent in the enhanced operation map imply both (a) proper base change and (b) the projection formula for these functors, in a homotopy coherent way.  The required enhanced operation map is built in stages. 

\begin{enumerate}

\item[(1)] In the first stage it is observed that an enhanced operation map (EOM) exists for (schemes, all morphisms, $\set{1}$). This is essentially just the statement that pullbacks $f^*\from \D(Y)\to \D(X)$ preserve symmetric monoidal structure.  It is a special case of a general theorem on enhanced operations for ringed topoi \cite[\S 2.2]{LZ}.  
We restrict this to an EOM for (qc separated schemes, locally finite type morphisms, $\set{1}$)

\item[(2)] In the next stage, we apply the fact that every morphism of qc separated schemes $f\from X\to Y$ which is locally of finite type can be factored as $p\circ j$, where $j$ is an open immersion and $p$ is proper.  A gluing technique \cite[Theorem 5.4]{LZ2} allows us to extend the EOM for (schemes, locally finite type morphisms, $\set{1}$) to (qc separated schemes, proper morphisms, local isomorphisms,$\set{1,2}$) and even to (qc separated schemes, proper morphisms, local isomorphisms, all morphisms, $\set{1,2,3}$).  So far we are still only encoding structures associated to the pullback functor $f^*$, but in a way that ``remembers'' all possible factorizations of $f$ as $p\circ j$.

\item[(3)] Heuristically, if $f=p\circ j$ as above then $Rf_!$ should be defined as $Rp_*\circ Rj_!$, where $Rp_*$ is right adjoint to $p^*$ and $Rj_!$ is left adjoint to $j^*$.  \cite[Proposition 1.4.4]{LZ2} is the abstract input required to pass from an EOM on (qc separated  schemes, proper morphisms, local isomorphisms, all, $\set{1,2,3}$) to an EOM on (qc separated  schemes, proper morphisms, local isomorphisms, all, $\set{3}$), i.e., the arrows have been reversed for the first two classes of morphisms.

\item[(4)] The same gluing technique as in (2) applied in reverse allows us to transfer the EOM to (qc separated schemes, locally finite type morphisms, all, $\set{2}$). By now, we have an $\infty$-categorical enhancement of $Rf_!$ for $f$ a morphism of coproducts of qcqs schemes which is locally of finite type.

\item[(5)] The ``DESCENT'' program developed in \cite{LZ} is a means of extending an EOM from one marked category $\CC$ to a marked overcategory $\tilde{\CC}$.  The input requires that for every object $X\in \tilde{\CC}$, there exists a marked morphism $Y\to X$ with $Y$ in $\CC$.  It is also required that the marked morphisms be of ``universal descent'' with respect to the EOM.

Repeated calls to DESCENT allow us to extend our EOM to the following domains:
\begin{enumerate}
\item From qc separated schemes to qs schemes,
\item From qs schemes to algebraic spaces,
\item From algebraic spaces to Artin stacks.
\end{enumerate}
The final output is the desired EOM on (Artin stacks, locally finite type morphisms, all morphisms, $\set{2}$).
\end{enumerate}

Our main result follows a similar strategy.  We highlight the major differences:

\begin{enumerate}

\item[(1)] In the first stage we have an EOM for (small coproducts of qcqs v-sheaves, all morphisms, $\set{1}$).  

\item[(2)] In the second stage, we start with the observation that if $f\from X\to Y$ is a morphism of qcqs v-sheaves which is representable in locally spatial diamonds and compactifiable of locally finite dim.trg, then there is a factorization $f = p\circ j$, where $j$ is an open immersion and $p$ is proper.  In fact there is a canonical compactification $j\from X\to \overline{X}^{/Y}$ \cite[Proposition 18.6]{ECoD}.  We encounter a substantial problem here:  we have little control over $\overline{X}^{/Y}$, and in fact $p$ is not necessarily representable in locally spatial diamonds.  We resolve this issue by introducing a larger auxiliary class of {\em prespatial diamonds}, which is preserved under passing to the canonical compactification rather by design.
The result is an EOM on (small coproducts of qcqs v-sheaves, proper morphisms representable in prespatial diamonds, separated local isomorphisms, $\set{1,2}$).  Let us emphasize that in this step, we make heavy use of the main cohomological results in \cite{ECoD}.

\item[(3)-(4)] These steps are similar to those in the scheme setting.  The result is an EOM on (small coproducts of qcqs v-sheaves, morphisms representable in locally spatial diamonds and compactifiable of locally finite dim.trg, $\set{2}$).

\item[(5)] The DESCENT program is applied repeatedly to extend the EOM to the following domains:
\begin{enumerate}
\item From small coproducts of separated spatial diamonds to quasiseparated locally separated locally
spatial diamonds,
\item From quasiseparated locally separated locally spatial diamonds to locally separated locally
spatial diamonds.
\item From locally separated locally spatial diamonds to decent v-stacks (with fine morphisms).
\end{enumerate}
The final output is the EOM which encodes the operation $Rf_!$ for fine morphisms between v-stacks;  this is what is necessary to prove Theorem \ref{mainthm}.
\end{enumerate}

%

\subsection{Comments and conventions}

Most readers of this article should simply take Theorem \ref{mainthm} as a black box. However, for the scrupulous reader, we recommend having \cite{LZ} and \cite{LZ2} close at hand. Not only will we heavily use the machinery introduced there, but we will borrow much notation from these papers, sometimes without comment. 

We need to heavily use $\infty$-categorical techniques. As in \cite{LZ, LZ2}, we use Lurie's model: an $\infty$-category is a simplicial set satisfying the weak Kan condition.  We often conflate ordinary categories with $\infty$-categories by identifying a category $\mathcal{C}$ with its nerve $\mathrm{N}(\mathcal{C})$. Likewise, we often conflate (2,1)-categories with $\infty$-categories by identifying a $(2,1)$-category $\mathcal{C}$ with its Duskin nerve, which we again denote $\mathrm{N}(\mathcal{C})$ (see \cite[Tag 009P]{Kerodon} for some discussion of this notion). We will usually omit the nerve functor from our notation (in a departure from the convention in \cite{LZ}).

\subsection{Acknowledgments}

DH would like to thank Johan de Jong, Yifeng Liu, Jacob Lurie, Lucas Mann, Peter Scholze, and Weizhe Zheng for some extremely helpful conversations related to this material. DH is also very grateful to his wife for her exceptional patience during several unusually intense periods of work on this project.

\section{Statement of main result}

\subsection{The notion of an enhanced operation map}
\label{SectionEOMaps}

A marked $\infty$-category is a pair $(\CC,\FF)$, where
$\mathcal{C}$ is a geometric $\infty$-category \cite[Definition 4.1.3]{LZ}, and $\FF$ is a set of morphisms of $\CC$ stable under composition, arbitrary pullback, and small coproducts.
The reader should imagine that $\CC$ is (the nerve of) some
ordinary category of geometric significance: the category of small
coproducts of quasicompact separated schemes, the category of locally
spatial diamonds, etc. Consider the following scenario, which is typical
of what one sees in a six-functor formalism.

\begin{scenario}\label{Scenario} Let $(\CC,\FF)$ be a marked $\infty$-category. 
\begin{enumerate}
\item[(1)] For all objects $X\in\mathcal{C}$, we have an associated closed
symmetric monoidal stable $\infty$-category $\mathcal{D}(X)$, which
the reader should imagine as the derived ($\infty$-)category of sheaves
on some ringed site associated with $X$. Moreover, we have an internal
hom bifunctor $R\mathscr{H}\mathrm{om}(-,-)$ such that $R\mathscr{H}\mathrm{om}(B,-)$
is right-adjoint to $-\otimes B$ for all $B\in\mathcal{D}(X)$.
\item[(2)] For any morphism $f:X\to Y$ we have a symmetric monoidal pullback
functor $f^{\ast}:\mathcal{D}(Y)\to\mathcal{D}(X)$.
\item[(3)] For any morphism $p:X\to Y$ with $p\in\mathcal{F}$ we have an exceptional
pushforward functor $p_{!}: \mathcal{D}(X) \to \mathcal{D}(Y)$.
\item[(4)] The functors $f^{\ast}$ and $p_{!}$ commute with all direct sums,
and therefore admit right adjoints $f_{\ast}$ resp. $p^{!}$.
\item[(5)] For any cartesian square
\begin{equation}
\label{EqCartesianSquare}
\xymatrix{X'\ar[r]^{g}\ar[d]_{q} & X\ar[d]^{p}\\
Y'\ar[r]_{f} & Y
}
\end{equation}
with $p\in\mathcal{F}$, there is a proper base change isomorphism
$f^{\ast}p_{!}\simeq q_{!}g^{\ast}$.
\item[(6)] There is a projection formula isomorphism $p_{!}A\otimes B\simeq p_{!}(A\otimes p^{\ast}B)$ for $A\in \D(X)$, $B\in \D(Y)$.
\end{enumerate}
\end{scenario}
We now consider the following problem:
\begin{prob}
Give a sensible way of cleanly encoding all of the structures in Scenario \ref{Scenario},\emph{
together} with all of their expected higher coherences with respect to composition.
\end{prob}

In this section we spell out one solution to this problem, closely
following the ideas of \cite{LZ}, namely the notion of an \emph{enhanced
operation map.} Before making a formal definition, we recall two $\infty$-categorical constructions from \cite{HA} and \cite{LZ2}.   The first involves symmetric monoidal structures on general $\infty$-categories, and the second has to do with the way we will encode base change isomorphisms on the $\infty$-categorical level.

\begin{cons}[Symmetric monoidal $\infty$-categories and commutative algebra objects] \label{ConsSymmetricMonoidal}

We review the definition of symmetric monoidal $\infty$-category \cite[Definition 2.0.0.7]{HA}.  Let $\Fin$ be the category of pointed finite sets, with objects $\class{n}=\set{\ast,1,\dots,n}$, and where the morphisms $\class{m}\to\class{n}$ are those functions which preserve $\ast$.   (Equivalently, $\Fin$ is the category of finite sets where the morphisms are partially defined functions.) 

A {\em symmetric monoidal $\infty$-category} is a coCartesian fibration of simplicial sets $p\from \CC^{\otimes} \to \N(\Fin)$, satisfying a certain condition ($\ast$).  For $i=1,\dots,n$, let $\rho^i\from \class{n}\to \class{1}$ be the unique morphism with $(\rho^i)^{-1}(1)=\set{i}$.  Since $p$ is a coCartesian fibration, $\rho^i$ induces a functor $\rho^i_!\from \CC^{\otimes}_{\class{n}}\to \CC^{\otimes}_{\class{1}}$.  The condition ($\ast)$ is that the product of the $\rho^i_!$ is an equivalence $\CC^{\otimes}_{\class{n}}\to (\CC^{\otimes}_{\class{1}})^n$ for all $n$. We write $\CC=\CC^{\otimes}_{\class{1}}$ and call it the underlying $\infty$-category of $\CC^{\otimes}$.  We may sometimes start with an $\infty$-category $\CC$, and say that $\CC^{\otimes}$ constitutes a symmetric monoidal structure on $\CC$, the map $p\from \CC^{\otimes}\to \N(\Fin)$ being understood.

In this situation we have a functor $\underline{\;\;}\otimes \underline{\;\;} \from \CC\times \CC \isom \CC^{\otimes}_{\class{2}} \to \CC$, where the last map is induced from the unique active map $\alpha\from \class{2}\to \class{1}$;  this is the composition law for $\CC$.  The rest of the structure of $\CC^{\otimes}$ furnishes isomorphisms witnessing the fact that $\otimes$ is unital, commutative, and associative, and it also encodes all higher coherences among those isomorphisms.

A {\em commutative algebra object} of a symmetric monoidal $\infty$-category $\CC^{\otimes}$ is a section $s\from \N(\Fin)\to \CC^{\otimes}$ to the functor $p$.  Let $X=s(\class{1})$;  then the unique active map $\alpha\from \class{2}\to\class{1}$ induces a composition law $m\from X\otimes X\to X$.  Once again, the rest of the structure of $s$ furnishes isomorphisms encoding the fact that $m$ is unital, commutative, and associative.  The commutative algebra objects of $\CC^{\otimes}$ form an $\infty$-category $\CAlg(\CC)$ \cite[Definition 2.1.3.1]{HA}. 

If $\CC$ is any $\infty$-category admitting finite products (resp., coproducts), then $\CC$ admits a  canonical symmetric monoidal structure $\CC^\times$ (resp., $\CC^{\amalg}$) for which the composition law is the product (resp.,  the coproduct), see \cite[Construction 2.4.1.4]{HA} (resp. \cite[Construction 2.4.3.1]{HA}).  For instance, the objects of $\CC^{\amalg}_{\class{n}}$ are $n$-tuples $(X_1,\dots,X_n)$ of objects of $\CC$, and a 1-morphism $(X_1,\dots,X_m)\to (Y_1,\dots,Y_n)$ in $\CC^{\amalg}$ consists of a morphism $\alpha\from \class{m}\to\class{n}$ together with a 1-morphism $X_i\to Y_{\alpha(i)}$ for each $1\leq i\leq m$ with $\alpha(i)\neq \ast$.  In particular this forces $X_1\otimes X_2\isom X_1\coprod X_2$. 

If $\CC$ is an $\infty$-category admitting finite products, we write $\CAlg(\CC)=\CAlg(\CC^\times)$.  In particular $\CAlg(\Cat_\infty)$ is the $\infty$-category of symmetric monoidal $\infty$-categories. 

There is a canonical functor $\CC\to \CAlg(\CC^{\amalg})$ which assigns to an object $A$ the commutative algebra structure for which the composition law is the obvious map $A\coprod A\to A$.  In fact $\CC^{\amalg}$ is the universal symmetric monoidal $\infty$-category $\mathcal{D}$ admitting a functor $\CC\to \CAlg(\mathcal{D})$, see \cite[Theorem 2.4.3.18]{HA}.  That is, for any such $\mathcal{D}$, there is an equivalence betwen monoidal functors $T\from \CC^{\coprod} \to \mathcal{D}$ and functors $T^{\otimes} \from \CC\to \CAlg(\mathcal{D})$.

In our desired application, $\CC$ is (the nerve of) some category of geometric objects, and we wish to assign to each object $X$ of $\CC$ a symmetric monoidal category $\infty$-category $\D(X)$ in a coherent manner, such that morphisms $X\to Y$ induce functors $\D(Y)\to \D(X)$.  This may be accomplished by constructing a functor $\CC^{\op}\to\CAlg(\Cat_\infty)$, which is equivalent to constructing a monoidal functor $\CC^{\op,\amalg} \to \Cat_\infty$.
\end{cons}

\begin{cons} Let $\CC$ be a category, let $k\geq 2$ be an integer, let $I\subset \set{1,\dots,k}$ be a subset, and let $\E_1,\dots,\E_k$ be sets of morphisms of $\CC$, each containing every identity morphism in $\CC$.  The {\em restricted multisimplicial nerve} $\delta_{k,I}^\ast\N(\CC)^{\cart}_{\E_1,\dots,\E_k}$ is a simplicial set constructed in \cite{LZ}.  For the moment we will only need it in the case $k=2$, $I=\set{2}$.  The 0-simplices of $\delta_{2,\set{2}}^\ast \N(\CC)_{\E_1,\E_2}$ are simply the objects of $\CC$.  The 1-simplices are cartesian squares
\begin{equation}
\label{EqDelta1Simplex}
\xymatrix{
c_{01} \ar[d]\ar[r] &  c_{00}  \ar[d] \\
c_{11} \ar[r] & c_{10}
}
\end{equation}
where the vertical (resp., horizontal) arrows lie in $\E_1$ (resp., $\E_2$); the vertices of this edge are $c_{00}$ and $c_{11}$.\footnote{The subscript $\set{2}$ in $\delta_{2,\set{2}}^\ast$ controls the shape of the diagram in \eqref{EqDelta1Simplex}, which resembles the diagram in \eqref{EqCartesianSquare}.  If instead $I=\emptyset$ the columns of \eqref{EqDelta1Simplex} would be transposed.}   The 2-simplices of $\delta_{2,\set{2}}^\ast \N(\CC)_{\E_1,\E_2}$ are diagrams
\begin{equation}
\label{EqDelta2Simplex}
\xymatrix{
c_{02} \ar[d] \ar[r] & c_{01} \ar[d] \ar[r] & c_{00} \ar[d] \\
c_{12} \ar[d] \ar[r] & c_{11} \ar[d] \ar[r] & c_{10} \ar[d] \\
c_{22} \ar[r] & c_{21} \ar[r] & c_{20}
}
\end{equation}
where each square is cartesian, and where once again the 
vertical (resp., horizontal) arrows lie in $\E_1$ (resp., $\E_2$).  The three edges of this 2-simplex are given by the upper-right square, the lower-left square, and the outer square in \eqref{EqDelta2Simplex}, respectively.  Given a 1-simplex corresponding to the diagram in \eqref{EqDelta1Simplex}, its degenerate 2-simplices are obtained by placing it in the upper-right or lower-left of a diagram like in \eqref{EqDelta2Simplex}, and ``filling in'' the rest of the diagram using identity morphisms.  Higher simplices, face maps and degeneracy maps are definted similarly.
\end{cons}

We can now describe the machinery of enhanced operation maps built in \cite{LZ2}.  Suppose given a marked $\infty$-category $(\mathcal{C},\mathcal{F})$ as in the beginning of \ref{SectionEOMaps}. We then have
an $\infty$-category $\mathcal{C}_{\amalg}=(\mathcal{C}^{\mathrm{op},\amalg})^{\mathrm{op}}$
equipped with a map of simplicial sets $\mathcal{C}_{\amalg}\to\mathrm{N}(\mathcal{F}\mathrm{in}_{\ast})^{\mathrm{op}}$
\cite[p. 297]{HA}. An object of $\mathcal{C}_{\amalg}$ is a pair $\left\langle n\right\rangle \in\mathcal{F}\mathrm{in}_{\ast}$
together with a sequence $(X_{1},\dots,X_{n})$ of objects in $\mathcal{C}$.
A morphism $f$ in $\mathcal{C}_{\amalg}$ from $(X_{1},\dots,X_{n})$
to $(Y_{1},\dots,Y_{m})$ consists of a map of pointed sets $\alpha:\left\langle m\right\rangle \to\left\langle n\right\rangle $   
together with a sequence of morphisms $\left\{ X_{\alpha(i)}\to Y_{i}\right\} _{i\in\alpha^{-1}\left\langle n\right\rangle ^{\circ}}$
in $\mathcal{C}$. 
Note that $\mathcal{F}$ induces a marking on $\mathcal{C}_{\amalg}$ by
taking those morphisms for which $\alpha:\left\langle m\right\rangle \to\left\langle m\right\rangle $
is the identity map and all the associated morphisms $\{X_{i}\to Y_{i}\}_{1\leq i\leq m}$
lie in $\mathcal{F}$. In the terminology of \cite{LZ}, these are the
edges of $\mathcal{C}_{\amalg}$ which \emph{statically belong to
$\mathcal{F}$}. By abuse of notation, we will also denote this marking
by $\mathcal{F}$. Note that the fiber of $\mathcal{C}_{\amalg}$
over $\left\langle 1\right\rangle $ is just $\mathcal{C}$, on which
the marking of $\mathcal{C}_{\amalg}$ restricts to the original marking
$\mathcal{F}$, so this abuse should cause no confusion.

Now, as in \cite{LZ2} we can form the simplicial set $\delta_{2,\{2\}}^{\ast}(\mathcal{C}_{\amalg})_{\mathcal{F},\mathrm{all}}^{\mathrm{cart}}$.
A 0-simplex is just an object $(X_{1},\dots,X_{n})$ of $\mathcal{C}_{\amalg}$.
\emph{If} $\mathcal{C}$ admits finite products, then a 1-simplex
consists of a map of pointed sets $\alpha:\left\langle m\right\rangle \to\left\langle n\right\rangle $
together with a diagram

\[
\xymatrix{(X_{1}',\dots,X_{n}')\ar[r]\ar[d] & (Y_{1}',\dots,Y_{m}')\ar[d]\\
(X_{1},\dots,X_{n})\ar[r] & (Y_{1},\dots,Y_{m})
}
\]
where the vertical edges statically belong to $\mathcal{F}$, the
horizontal edges are morphisms in $\mathcal{C}_{\amalg}$ lying over
$\alpha$, and for all $j\in\left\langle n\right\rangle ^{\circ}$
the induced diagrams
\[
\xymatrix{X_{j}'\ar[r]\ar[d] & \prod_{i\in\alpha^{-1}(j)}Y_{i}'\ar[d]\\
X_{j}\ar[r] & \prod_{i\in\alpha^{-1}(j)}Y_{i}
}
\]
are cartesian.

Given this setup, an enhanced operation map is a functor
\[
_{\mathcal{C}}\mathrm{EO}:\delta_{2,\{2\}}^{\ast}(\mathcal{C}_{\amalg})_{\mathcal{F},\mathrm{all}}^{\mathrm{cart}}\to\mathcal{C}\mathrm{at}_{\infty}
\]
satisfying various properties. To state these properties, we introduce
some notation attached to such a functor.
\begin{itemize}
\item Any object $X\in\mathcal{C}$ defines a 0-simplex of $\delta_{2,\{2\}}^{\ast}(\mathcal{C}_{\amalg})_{\mathcal{F},\mathrm{all}}^{\mathrm{cart}}$
lying over $\left\langle 1\right\rangle \in\mathcal{F}\mathrm{in}_{\ast}$,
and we set $\mathcal{D}(X):=\mathrm{_{\mathcal{C}}EO}(X)$.
\item Restriction to the ``all'' direction defines a functor 
\[
_{\mathcal{C}}\mathrm{EO}^{\mathrm{I}}:\mathcal{C}^{\mathrm{op},\amalg}\to\mathcal{C}\mathrm{at}_{\infty},
\]
and further restriction to the fiber over $\left\langle 1\right\rangle $
defines a functor 
\[
_{\mathcal{C}}\mathrm{EO}^{\ast}:\mathcal{C}^{\mathrm{op}}\to\mathcal{C}\mathrm{at}_{\infty}.
\]
Given any morphism $f:X\to Y$ in $\mathcal{C}$, we write $f^{\ast}:\mathcal{D}(Y)\to\mathcal{D}(X)$
for the functor given by the image of $f$ under $_{\mathcal{C}}\mathrm{EO}^{\ast}$.
\item Restriction of $_{\mathcal{C}}\mathrm{EO}$ to the fiber over $\left\langle 1\right\rangle $
defines a functor
\[
_{\mathcal{C}}\mathrm{EO}_{!}^{\ast}:\delta_{2,\{2\}}^{\ast}(\mathcal{C}{}_{\mathcal{F},\mathrm{all}}^{\mathrm{cart}})\to\mathcal{C}\mathrm{at}_{\infty},
\]
and further restriction to the $\mathcal{F}$ direction defines a
functor
\[
_{\mathcal{C}}\mathrm{EO}_{!}:\mathcal{C}_{\mathcal{F}}\to\mathcal{C}\mathrm{at}_{\infty}.
\]
Given any morphism $p:X\to Y$ in $\mathcal{C}$ with $p\in\mathcal{F}$,
we write $p_{!}:\mathcal{D}(X)\to\mathcal{D}(Y)$ for the functor
given by the image of $p$ under $_{\mathcal{C}}\mathrm{EO}_{!}$.
\end{itemize}
\begin{defn}
\label{DefnEnhancedOperation}
Notation as above, the functor 
\[
_{\mathcal{C}}\mathrm{EO}:\delta_{2,\{2\}}^{\ast}(\mathcal{C}_{\amalg})_{\mathcal{F},\mathrm{all}}^{\mathrm{cart}}\to\mathcal{C}\mathrm{at}_{\infty}
\]
 is an \emph{enhanced operation map }if the following two conditions
are satisfied.

\begin{enumerate}
\item The functor $_{\mathcal{C}}\mathrm{EO}^{\mathrm{I}}$ is a weak
Cartesian structure \cite[Definition 2.4.1.1]{HA}, and the induced functor (see Construction \ref{ConsSymmetricMonoidal})\footnote{See \cite[Remark 1.5.6]{LZ} for this notation.}
\[
\left(_{\mathcal{C}}\mathrm{EO}^{\mathrm{I}}\right)^{\otimes}:\mathcal{C}^{\mathrm{op}}\to\mathrm{CAlg}(\mathcal{C}\mathrm{at}_{\infty})
\]
factors through $\mathrm{CAlg}(\mathcal{C}\mathrm{at}_{\infty})_{\mathrm{pr,st,cl}}^{\mathrm{L}}$
and sends small products in $\mathcal{C}^{\mathrm{op}}$ (i.e. small
coproducts in $\mathcal{C}$) to products.  Here $\CAlg(\Cat_\infty)^L_{\mathrm{pr,st,cl}}$ is the category appearing in \cite[Definition 1.5.2]{LZ}.  It is the subcategory of $\CAlg(\Cat_\infty)$ whose objects are symmetric monoidal $\infty$-categories which are:
\begin{itemize}
\item presentable \cite[Definition 5.5.0.1]{HTT},
\item stable \cite[Definition 1.1.1.9]{HA}, and
\item closed \cite[Definition 4.1.1.17]{HA},
\end{itemize}
and for which the morphisms $\CC^\otimes\to \D^\otimes$ are left adjoints.

\item  The functor $_{\mathcal{C}}\mathrm{EO}_{!}^{\ast}\from \delta_{2,\set{2}}^{\ast}(\CC_{\mathcal{F},\mathrm{all}}^{\mathrm{cart}})\to \Cat_\infty$ factors through
$\mathcal{P}\mathrm{r}_{\mathrm{st}}^{\mathrm{L}}$.  Here, $\PPr_{\st}^{\LL}\subset \Cat_\infty$ is the subcategory whose objects are presentable stable $\infty$-categories, and whose morphisms are left adjoint functors.

\end{enumerate}
\end{defn}

An enhanced operation map is an extremely dense piece of structure,
and the full content of this structure is probably opaque at first
glance. However, we claim this notion gives a reasonable solution
to Problem 2.1. This claim is justified as follows:

\begin{itemize}
\item Condition (1) in Definition \ref{DefnEnhancedOperation} has the following consequences.  
For all $X\in\mathcal{C}$ we have a presentable stable closed symmetric monoidal $\infty$-category $\D(X):=(_{\CC}\EO^I)^{\otimes}(X)$.  Let us note here that $_{\CC}\EO(X,X)=\D(X)\times \D(X)$, and that evaluating $_{\CC}\EO$ on the 1-simplex
\[
\xymatrix{
X \ar[r]\ar[d] & (X,X) \ar[d] \\
X \ar[r] & (X,X) 
}
\]
defines a functor $\mathcal{D}(X)\times\mathcal{D}(X)\to\mathcal{D}(X)$, which is none other than the symmetric monoidal structure on $\D(X)$.  
Since the symmetric monoidal structure on $\D(X)$ is closed, the functor
$A\mapsto A\otimes B$ commutes with all colimits and therefore admits a right
adjoint, giving the desired internal hom. This verifies item (1) of Scenario \ref{Scenario}.  The construction of $f^*$ and $p_!$ we have given above verifies items (2) and (3).

\item Condition (2) in Definition \ref{DefnEnhancedOperation} implies that $f^{\ast}$
(for arbitrary morphisms $f$) and $p_{!}$ (for morphisms $p\in\mathcal{F}$)
are morphisms in $\mathcal{P}\mathrm{r}_{\mathrm{st}}^{\mathrm{L}}$,
and therefore admit right adjoints. This verifies item (4) of Scenario \ref{Scenario}.

\item Applying $_{\mathcal{C}}\mathrm{EO}_{!}^{\ast}$ to a 1-simplex
\[
\xymatrix{X'\ar[r]^{g}\ar[d]_{q} & X\ar[d]^{p}\\
Y'\ar[r]_{f} & Y
}
\]
of $\delta_{2,\{2\}}^{\ast}\mathcal{C}{}_{\mathcal{F},\mathrm{all}}^{\mathrm{cart}}$
defines a functor $F\from \mathcal{D}(X)\to\mathcal{D}(Y')$.
Now the 2-simplices
\[
\xymatrix{
X'\ar[r] \ar[d]& X \ar[r] \ar[d] & X\ar[d] \\
Y' \ar[r] \ar[d]& Y \ar[r] \ar[d] & Y\ar[d] \\
Y' \ar[r] & Y \ar[r] &    Y 
}\text{   and   }
\xymatrix{
X'\ar[r] \ar[d]& X' \ar[r] \ar[d] & X\ar[d] \\
X' \ar[r] \ar[d]& X' \ar[r] \ar[d] & X\ar[d] \\
Y' \ar[r] & Y' \ar[r] &    Y 
}
\]
witness isomorphisms $F\cong f^*p_!$ and $F\cong q_!g^*$, respectively.  In particular, we get a proper base change equivalence $f^{\ast}p_{!}\simeq q_{!}g^{\ast}$,
verifying item (5) of Scenario \ref{Scenario}.

\item Suppose $\mathcal{C}$ admits finite products. Then for any morphism $p\from X\to Y$
in $\mathcal{F}$, there is a 1-simplex
\[
\xymatrix{(X)\ar[r]\ar[d]_{p} & (X,Y)\ar[d]^{(p,\mathrm{id})}\\
(Y)\ar[r] & (Y,Y)
}
\]
of $\delta_{2,\{2\}}^{\ast}(\mathcal{C}_{\amalg})_{\mathcal{F},\mathrm{all}}^{\mathrm{cart}}$
lying over the unique active map $\alpha:\left\langle 2\right\rangle \to\left\langle 1\right\rangle $,
and applying $_{\mathcal{C}}\mathrm{EO}$ to this 1-simplex defines
a functor $G\from \mathcal{D}(X)\times\mathcal{D}(Y)\to\mathcal{D}(Y)$.  The 2-simplices
\[
\xymatrix{
X \ar[r]\ar[d] & (X,X) \ar[r] \ar[d] & (X,Y)\ar[d] \\
Y \ar[r]\ar[d]  & (Y,Y) \ar[r]\ar[d] & (X,Y)\ar[d]\\
Y \ar[r] & (Y,Y) \ar[r] & (Y,Y)
}
\text{ and }
\xymatrix{
X \ar[r]\ar[d] & X \ar[r]\ar[d]  & (X,Y) \ar[d] \\
X \ar[r]\ar[d] & X \ar[r] \ar[d] & (X,Y)\ar[d] \\
Y \ar[r] & Y \ar[r] & (Y,Y)
}
\]
witness isomorphisms $G(A,B)\cong p_!A \otimes B$ and $G(A,B)\cong p_!(A\otimes p^*B)$, respectively.  In particular, we get the projection formula, verifying item (6) of Scenario \ref{Scenario}.
\end{itemize}

\subsection{The main theorem}

Using the language and notation of the previous section, we can now
state the main technical theorem proved in this paper.
\begin{thm}\label{mainthmtechnical}
Consider the marked $\infty$-category $(\mathcal{V}\mathrm{stk}^{\mathrm{dct}},\mathcal{F})$
where $\mathcal{V}\mathrm{stk}^{\mathrm{dct}}$ denotes the category
of decent v-stacks, and $\mathcal{F}$ is the class of fine morphisms.
Fix a ring $\Lambda$ killed by some integer prime to $p$. Then there
is an enhanced operation map
\[
_{\mathcal{V}\mathrm{stk}^{\mathrm{dct}}}\mathrm{EO}:\delta_{2,\{2\}}^{\ast}(\mathcal{V}\mathrm{stk}_{\amalg}^{\mathrm{dct}})_{\mathcal{F},\mathrm{all}}^{\mathrm{cart}}\to\mathcal{C}\mathrm{at}_{\infty}
\]
such that:
\begin{enumerate}
\item $\mathcal{D}(X)=\mathcal{D}_{\mathrm{\acute{e}t}}(X,\Lambda)$ with its natural symmetric monoidal structure for all $X\in\mathcal{V}\mathrm{stk}^{\mathrm{dct}}$.

\item $f^{\ast}$ coincides with the pullback functor on $\mathcal{D}_{\mathrm{\acute{e}t}}$
constructed in \cite{ECoD} for all morphisms $f$. 

\item For morphisms $f\in\mathcal{F}$ which are separated and
representable in locally spatial diamonds, $f_{!}$ coincides with
the functor $Rf_{!}$ constructed in \cite{ECoD}.
\end{enumerate}
\end{thm}

Theorem \ref{mainthmtechnical} implies Theorem \ref{mainthm}. 
The functors $Rf_!$ required by Theorem \ref{mainthm} are obtained from the functors $f_!$ coming from the enhanced operation functor by passing from $\D_{\et}(X,\Lambda)$ to its homotopy category $D_{\et}(X,\Lambda)$. 
In light of the discussion in the previous section, $Rf_!$ satisfies the projection formula and proper base change.  Since $_{\CC}\EO^*_!$ takes values in $\Pr_{\st}^{\LL}$, the functor $Rf_!$ is a left adjoint.

\section{Enhanced operations for qcqs v-sheaves}

\subsection{Prespatial diamonds}

In the theory of \cite{ECoD}, (locally) spatial diamonds play a central role. There are several justifications for this centrality: they capture most diamonds of practical interest, they have excellent categorical properties, and their \'etale cohomology admits some a priori control in terms of simple ``dimensional'' invariants. However, from the perspective of this paper, they suffer one serious defect: spatial diamonds are not known to be stable under the formation of \emph{canonical compactifications}. In particular, if $f:X \to Y$ is a map which is compactifiable and representable in spatial diamonds, it is not known whether $f$ admits a factorization $X \overset{j}{\to} \overline{X}  \overset{p}{\to} Y$ where $j$ is an open immersion and $p$ is proper and representable in spatial diamonds. This lacuna in our knowledge is a serious obstacle to implementing the categorical gluing arguments from \cite{LZ2}.

In this section, we overcome this difficulty by slightly enlarging the category of (locally) spatial diamonds. More precisely, we introduce a notion of (locally) \emph{prespatial} diamonds. On one hand, we will see that these gadgets enjoy all the same good properties as (locally) spatial diamonds, including the same crucial a priori control on \'etale cohomology. On the other hand, they are stable under passing to canonical compactifications, essentially by design.\footnote{Finding the ``right'' generalization of locally spatial diamonds with all of these properties involved several years of trial and error, and turned out to be the main bottleneck in the completion of this project. In the end, we arrived at the definition presented here through extended meditation on the proof of \cite[Theorem 22.5]{ECoD}. We will revisit that proof in the argument for Proposition \ref{prespatialkeybound} below.}

We now turn to the key definitions.
\begin{defn}
A qcqs diamond $X$ is \emph{prespatial} if there exists a spatial
subdiamond $X_{0}\subset X$ such that $X_{0}(K,\mathcal{O}_{K})=X(K,\mathcal{O}_{K})$
for all perfectoid fields $K$.

A diamond $X$ is \emph{locally prespatial} if there exists a locally
spatial subdiamond $X_{0}\subset X$ such that the inclusion map $X_{0}\to X$
is quasicompact and $X_{0}(K,\mathcal{O}_{K})=X(K,\mathcal{O}_{K})$
for all perfectoid fields $K$.

A pair $(X_0,X)$ satisfying the above conditions will be called a \emph{(locally) prespatial pair}.
\end{defn}

Note that in the prespatial case, the inclusion $X_{0}\to X$ is automatically
qcqs. However, we do not impose any further conditions on the map
$X_{0}\to X$. Note also that $X_{0}$ is far from unique in general:
if $X_{0}'\subset X$ and $X_{0}''\subset X$ both verify the definition
of a (locally) prespatial diamond, then so does $X_{0}'\times_{X}X_{0}''$.
More generally, the following lemma holds.
\begin{lem} \label{pair base change}
Let $(X_0,X)$ be a (locally) prespatial pair, let $(Y_0,Y)$ be a locally prespatial pair,
and let $f \colon X \to Y$ be a morphism of diamonds.
Then $(X_0 \times_Y Y_0,X)$ is also a (locally) prespatial pair.
\end{lem}
\begin{proof}
Since $Y_0 \to Y$ is a qc injection, its base change $X_0 \times_Y Y_0 \to X_0$ is also a qc injection, so $X_0 \times_Y Y_0$ is locally spatial by
\cite[Proposition 11.20]{ECoD}. Moreover, $X_0 \times_Y Y_0 \to X$ is qc since it factors as $X_0 \times_Y Y_0 \to X_0 \to X$ where both maps are qc.
If moreover $X_0$ is spatial, then $X_0 \times_Y Y_0$ is qcqs, hence spatial.
It is clear that for any perfectoid field $K$,
$(X_0 \times_Y Y_0)(K,\mathcal{O}_K)=X_0(K,\mathcal{O}_K)=X(K,\mathcal{O}_K)$. This gives the result.
\end{proof}

\begin{defn}
Let $f \colon X\to Y$ be a map of v-stacks. Say $f$ is
\emph{representable in (locally) prespatial diamonds}
if $f$ is 0-truncated and for all
maps $W\to Y$ with $W$ a (locally) prespatial diamond, $X\times_{Y}W$
is a (locally) prespatial diamond.
\end{defn}

\begin{lem}[Sanity Checks] 

\begin{enumerate}[label=(\roman*)]
\item A diamond $X$ is prespatial iff it is locally prespatial and qcqs. 
\item Any open subdiamond of a locally prespatial diamond is locally prespatial. 
\item Any (locally) spatial diamond is (locally) prespatial. 
\item If $f$ is representable in (locally) spatial diamonds, it is representable
in (locally) prespatial diamonds. 
\item Any fiber product of (locally) prespatial diamonds is a (locally)
prespatial diamond. 
\item Morphisms which are representable in (locally) prespatial diamonds
are stable under composition and base change.
\end{enumerate}
\end{lem}
\begin{proof}~
\begin{enumerate}[label=(\roman*)]
\item Suppose that $(X_0,X)$ is a prespatial pair.
Then $X_0 \to X$ is qc since $X_0$ is qc and $X$ is qs.  So $X$ is also locally
prespatial.  

Conversely, suppose that $(X_0,X)$ is a locally prespatial pair and
$X$ is qcqs.  Then $X_0$ is qcqs since $X$ is qcqs and
$X_0 \hookrightarrow X$ is a qc injection.  So $X_0$ is spatial.

\item If $(X_0,X)$ is a locally prespatial pair and $U$ is an open subdiamond of $X$, then
$U \times_X X_0$ is locally spatial by \cite[Proposition 11.19(ii)]{ECoD} (or Lemma \ref{pair base change}).
Since base change preserves the property of being a qc injection,
$(U \times_X X_0,U)$ is a locally prespatial pair.

\item If $X$ is (locally) spatial, then $(X,X)$ is a (locally) prespatial pair.

\item Let $f \colon X \to Y$ be a morphism that is representable in locally spatial diamonds, let $(Z_0,Z)$
be a locally prespatial pair, and suppose that we are given a map $Z \to Y$.  Then $(Z_0 \times_Y X,Z \times_Y X)$
is a prespatial pair.
Suppose $f$ is representable in spatial diamonds and $(Z_0,Z)$ is a prespatial
pair.  Then $f$, $Z_0$, and $Z$ are qcqs, so $Z_0 \times_Y X$ is qcqs,
hence spatial, and $Z \times_Y X$ is also qcqs.

\item Suppose we have (locally) prespatial pairs $(X_0,X)$, $(Y_0,Y)$, $(Z_0,Z)$, and a fiber product $X \times_Z Y$.
Then $X_0 \times_Z Z_0$ and $Y_0 \times_Z Z_0$ are (locally) spatial by Lemma \ref{pair base change}.
Hence $(X_0 \times_Z Z_0) \times_{Z_0} (Y_0 \times_Z Z_0)=X_0 \times_Z Y_0 \times_Z Z_0$ is (locally) spatial,
and $(X_0 \times_Z Y_0 \times_Z Z_0,X\times_Z Y)$ is a (locally) spatial pair.

\item This is clear from the definition.
\end{enumerate}
\end{proof}

\begin{warn}\label{locallyprespatialnotlocal} If $X$ is a diamond, it is not clear whether the
property of being locally prespatial can be checked locally on an
open cover of $X$. In light of this, the name is perhaps slightly
misleading.
\end{warn}
\begin{prop}\label{prespatialstable}
If $f:X\to Y$ is separated and representable in prespatial diamonds,
then $\overline{f}^{/Y}:\overline{X}^{/Y}\to Y$ is proper and representable
in prespatial diamonds.
\end{prop}

\begin{proof}
We can assume that $X$ and $Y$ are prespatial diamonds. Then $\overline{X}^{/Y}\to Y$
is proper by \cite[Prop. 18.7.(vii)]{ECoD}. In particular, $\overline{X}^{/Y}$
admits a separated map to a quasiseparated target and thus is quasiseparated.
Quasicompacity of $\overline{X}^{/Y}$ is clear. Finally, taking any
$X_{0}\subset X$ as in the definition of a prespatial diamond, we
have $\overline{X}^{/Y}(K,\mathcal{O}_{K})=X(K,\mathcal{O}_{K})=X_{0}(K,\mathcal{O}_{K})$
by the definition of $\overline{X}^{/Y}$, so $\overline{X}^{/Y}$
is prespatial.
\end{proof}
\begin{prop}\label{prespatialkeybound}
If $f:X\to Y$ is proper and representable in prespatial diamonds,
then $Rf_{\ast}:D_{\mathrm{\acute{e}t}}(X,\Lambda)\to D_{\mathrm{\acute{e}t}}(Y,\Lambda)$
has cohomological amplitude $\leq3\mathrm{dim.trg}f$.
\end{prop}

\begin{proof}
We follow the proof of \cite[Theorem 22.5]{ECoD} closely. Arguing as
in that proof, we can assume that $Y=\mathrm{Spd}(C,C^{+})$, $X$
is a separated prespatial diamond proper over $Y$, and $A\in D_{\mathrm{\acute{e}t}}(X,\Lambda)$
is concentrated in degree zero with $j^{\ast}A=0$. Here $j:f^{-1}(U)\to X$
is the natural open immersion, where $U\subset Y$ is the complement
of the unique closed point. 

Let $|X|^{h}$ be the maximal Hausdorff quotient of $|X|$, or equivalently
the maximal Hausdorff quotient of $|X\times_{\mathrm{Spd}(C,C^{+})}\mathrm{Spd}(C,\mathcal{O}_{C})|$.
As in \emph{loc. cit.,} we factor $f$ as $g:X\to\underline{|X|^{h}}\times Y$
followed by $h:\underline{|X|^{h}}\times Y\to Y$. As in \emph{loc.
cit.,} $g$ is proper and representable in spatial diamonds, and thus
$Rg_{\ast}$ has cohomological amplitude $\leq2\mathrm{dim.trg}f$. 

It remains to check that if $B\in D_{\mathrm{\acute{e}t}}(\underline{|X|^{h}}\times Y,\Lambda)$
is concentrated in degree zero and trivial on $\underline{|X|^{h}}\times U$,
then $R^{i}h_{\ast}B=0$ for all $i>\mathrm{dim.trg}f$. As in \emph{loc.
cit., }we can identify 
\[
R^{i}h_{\ast}B\cong H^{i}(\underline{|X|^{h}}\times Y,B)\cong H^{i}(|X|^{h},B|_{|X|^{h}\times{s}})
\]
(where we implicitly appeal to \cite[Proposition 22.7]{ECoD}). Thus
it remains to control the cohomology of abelian sheaves on $|X|^{h}$.
The key observation now is that for any choice of $X_{0}\subset X$
as in the definition of a prespatial diamond, we get an induced homeomorphism
\[|X_{0}\times_{\mathrm{Spd}(C,C^{+})}\mathrm{Spd}(C,\mathcal{O}_{C})|^{h}\cong|X|^{h}.\]
This follows from the general observation that if $i:U\to V$ is any injection of qcqs diamonds
which induces a bijection on $(K,\mathcal{O}_{K})$-points for all
perfectoid fields $K$, then $i$ induces a homeomorphism $|U|^{h}\cong|V|^{h}$.\footnote{To verify this, observe that $|U|^{h}\to|V|^{h}$ is a continuous
bijection by the assumption on $(K,\mathcal{O}_{K})$-points. But
any continuous bijection from a compact space to a Hausdorff space
is a homeomorphism.} Since $X_{0}$ is spatial, the Krull dimension of $|X_{0}\times_{\mathrm{Spd}(C,C^{+})}\mathrm{Spd}(C,\mathcal{O}_{C})|$
is bounded above by 
\begin{align*}
\mathrm{dim.trg}X_{0}\times_{\mathrm{Spd}(C,C^{+})}\mathrm{Spd}(C,\mathcal{O}_{C})/\mathrm{Spd}(C,\mathcal{O}_{C}) & \leq\mathrm{dim.trg}X_{0}/\mathrm{Spd}(C,C^{+})\\
 & \leq\mathrm{dim.trg}f.
\end{align*}
The desired bound on the cohomological dimension of $|X|^{h}$ now follows from the next lemma.
\end{proof}
\begin{lem}
Let $X$ be a separated spatial diamond, and let $d$ be the Krull
dimension of the spectral space $|X|$.  Then the cohomological dimension of the compact Hausdorff space
$|X|^{h}$ is $\leq d$.
\end{lem}

\begin{proof}

Note that the set of generizations of any point in $|X|$ forms a totally ordered chain, as this is true for any locally spatial diamond. Now, let $S$ be any spectral space in which the generizations of any point form a chain, and let $q:S\to S^{h}$ be the natural map
to the maximal Hausdorff quotient. We claim that in fact $R\Gamma(S^{h},\mathcal{F})\cong R\Gamma(S,q^{\ast}\mathcal{F})$
for any abelian sheaf $\mathcal{F}$ on $S^{h}$. In the case of interest
to us, $S=|X|$ has cohomological dimension $\leq d$ by Scheiderer's
theorem \cite{Scheiderer}, so this implies the desired result.

It's clearly enough to prove that $\mathcal{F}\cong Rq_{\ast}q^{\ast}\mathcal{F}$.
Let $x\in S^{h}$ be any point, and let $\tilde{x}\in S$ be the unique
maximal point in the fiber $q^{-1}(x)$, so $q^{-1}(x)=\overline{\{\tilde{x}\}}$.
Let $\mathcal{P}$ be the cofiltered set of all open neighborhoods
of $x$ in $S^{h}$, and let $\mathcal{N}$ be the cofiltered set
of all quasicompact open neighborhoods of $q^{-1}(x)$ in $S$. By
\cite[Lemma 8.1.5]{Hub96}, each of the collections $\{V\subset S\}_{V\in\mathcal{N}}$
and $\{q^{-1}(U)\subset S\}_{U\in\mathcal{P}}$ is a fundamental system
of neighborhoods of $q^{-1}(x)$ in $S$, and moreover 
\[
q^{-1}(x)=\cap_{V\in\mathcal{N}}V=\cap_{U\in\mathcal{P}}q^{-1}(U).
\]
Then
\begin{align*}
(R^{i}q_{\ast}q^{\ast}\mathcal{F})_{x}\cong & \mathrm{colim}_{U\in\mathcal{P}}H^{i}(q^{-1}(U),q^{\ast}\mathcal{F})\\
\cong & \mathrm{colim}_{V\in\mathcal{N}}H^{i}(V,q^{\ast}\mathcal{F})\\
\cong & H^{i}(q^{-1}(x),q^{\ast}\mathcal{F}).
\end{align*}
Since $q^{\ast}\mathcal{F}$ is constant on the fiber $q^{-1}(x)$,
we're reduced to showing that $R\Gamma(q^{-1}(x),A)\cong A$ for any
constant sheaf of abelian groups $A$ on $q^{-1}(x)$. This is an
easy exercise, using the fact that $q^{-1}(x)$ is a spectral space
with a unique maximal point in which the generalizations of any point
form a chain. (Precisely: If $T$ is a such a spectral space, and
$j:\eta\to T$ is the inclusion of the unique maximal point, then
$A\overset{\sim}{\to}Rj_{\ast}A$ for any constant sheaf of abelian
groups $A$.)
\end{proof}
\begin{cor}\label{dimensioncor}
If $f:X\to Y$ is separated and representable in prespatial diamonds
with $\mathrm{dim.trg}f<\infty$, then $Rf_{\ast}$ has cohomological
amplitude $\leq3\mathrm{dim.trg}f$. If moreover $f$ is compactifiable,
then $Rf_{!}$ has cohomological amplitude $\leq3\mathrm{dim.trg}f$.
\end{cor}

\begin{proof}
Factor $f$ as $\overline{f}^{/Y}\circ j$, where $j:X\to\overline{X}^{/Y}$
is the natural quasicompact injection. By Proposition \ref{prespatialstable}, $\overline{f}^{/Y}$
is proper and representable in prespatial diamonds with $\mathrm{dim.trg}\overline{f}^{/Y}=\mathrm{dim.trg}f$,
so $R\overline{f}_{\ast}^{/Y}$ has cohomological amplitude $\leq3\mathrm{dim.trg}f$
by Proposition \ref{prespatialkeybound}. Since $Rf_{!}=R\overline{f}_{\ast}^{/Y}\circ j_{!}$,
this implies the claim for $Rf_{!}$. Similarly, writing $Rf_{\ast}=R\overline{f}_{\ast}^{/Y} \circ Rj_\ast$, the desired bound for $Rf_\ast$ follows from the observation that $Rj_{\ast}$
has cohomological amplitude zero, which is a special case of the next lemma.
\end{proof}
\begin{lem}
If $j:X\to Y$ is a quasicompact injection of small v-stacks, then
$Rj_{\ast}$ has cohomological amplitude zero.
\end{lem}
\begin{proof}
By the first half of \cite[Prop. 17.6]{ECoD}, we can assume that $Y$
is a spatial diamond, in which case $X$ is also a spatial diamond.
This reduces us to \cite[Lemma 21.13]{ECoD}.
\end{proof}

With these results on the books, we now make another definition.
\begin{defn}
Let $f:X\to Y$ be a morphism of small v-stacks. We say $f$ is \emph{strongly
compactifiable }if it is compactifiable, representable in prespatial
diamonds, and locally of finite dim.trg. We say $f$ is\emph{ weakly
compactifiable }if it is compactifiable, representable in locally
prespatial diamonds, and $X$ admits an open cover $X=\cup X_{i}$
with each $X_{i}\to Y$ strongly compactifiable.
\end{defn}
Recall that by definition, a compactifiable morphism is necessarily 0-truncated and separated.

\begin{lem}\label{swsanity}
\begin{enumerate}[label=(\roman*)]
\item The property of being strongly resp. weakly compactifiable is stable under composition and base change.
\item A morphism $f$ is strongly compactifiable iff it is weakly compactifiable
and quasicompact. 
\item A morphism is proper and weakly compactifiable iff it is proper and strongly compactifiable.
\item Any strongly compactifiable morphism $f$ can be factored as $\overline{f}\circ j$
where $j$ is a quasicompact open immersion and $\overline{f}$ is
proper and strongly compactifiable.
\item If $f:X\to Y$ is strongly compactifiable, then $Rf_{\ast}$ and
$R\overline{f}_{\ast}^{/Y}$ satisfy base change on unbounded complexes
and commute with all colimits. Moreover, $Rf_{!}=R\overline{f}_{\ast}^{/Y}\circ j_{!}$
satisfies composability, base change, and the projection formula,
and commutes with all colimits.
\end{enumerate}
\end{lem}
\begin{proof}Parts (i)-(iv) are easy and left to the reader. Part (v) follows from the same arguments used in \cite{ECoD} for the case where $f$ is representable in spatial diamonds, using the cohomological dimension bounds from Corollary \ref{dimensioncor}.
\end{proof}
 
Note that in \cite{ECoD}, $Rf_!$ and $Rf^!$ are constructed exactly for morphisms $f$ which are weakly compactifiable \emph{and} representable in locally spatial diamonds. We now have the following basic claim.

\begin{schol}\label{weaklycompactifiableshriek}All constructions and results involving $Rf_!$ and $Rf^!$ established in \cite[\S 22-24]{ECoD} extend to the setting of weakly compactifiable morphisms.
\end{schol}

One simply repeats all arguments in \cite{ECoD} with extremely minor changes; we will not need the full generality of this claim, so we omit the details. However, let us analyze one particular case of this more general $Rf_!$ construction, which will be encoded in the enhanced operation map constructed in the next section. Suppose $f:X\to Y$ is a weakly compactifiable map, and that $X$ and $Y$ are coproducts
of qcqs v-sheaves, say with $X=\coprod_{i \in I} X_{i}$ where all $X_i$ are qcqs. There is a naturally
associated diagram
\[
\xymatrix{X=\coprod_{i \in I} X_{i}\ar[r]^{j}\ar[d]^{f} & \coprod_{i \in I} \overline{X_{i}}^{/Y}\ar[d]^{p}\\
Y & \coprod_{i \in I}Y\ar[l]^{h}
}
\]
where $p$ is proper and strongly compactifiable, $j$ is an open
immersion, and $h$ is a local isomorphism. In this notation, we have
a natural identification $Rf_{!} \cong h_{!}\circ Rp_{\ast}\circ j_{!}$.

\subsection{Enhanced operations for qcqs v-sheaves}

The main goal of this section is the following theorem.
\begin{thm}\label{seed}
Let $\mathcal{V}\mathrm{sh}^{\mathrm{qcqs}}$ denote the category
spanned by small coproducts of qcqs v-sheaves, and let $W$ denote the class of weakly compactifiable maps. Fix a ring $\Lambda$ killed
by some integer prime to $p$. Then there is an enhanced operations
map
\[
_{\mathcal{V}\mathrm{sh}^{\mathrm{qcqs}}}\mathrm{EO}:\delta_{2,\{2\}}^{\ast}(\mathcal{V}\mathrm{sh}^{\mathrm{qcqs}}_{\amalg})_{W,\mathrm{all}}^{\mathrm{cart}}\to\mathcal{C}\mathrm{at}_{\infty}
\]
with the following properties.

\emph{i. }$\mathcal{D}(X)=\mathcal{D}_{\mathrm{\acute{e}t}}(X,\Lambda)$ with its natural symmetric monoidal structure for all $X\in\mathcal{V}\mathrm{sh}^{\mathrm{qcqs}}$.

\emph{ii. }$f^{\ast}$ coincides with the pullback functor on $\mathcal{D}_{\mathrm{\acute{e}t}}$
constructed in \cite{ECoD} for all morphisms $f$. 

\emph{iii. }For morphisms $g\in W$, $g_{!}$ coincides with
the functor $Rg_{!}$ constructed in Scholium \ref{weaklycompactifiableshriek} and the discussion immediately afterwards.
\end{thm}

Our construction of $_{\mathcal{V}\mathrm{sh}^{\mathrm{qcqs}}}\mathrm{EO}$
closely follows the arguments in \cite[pp. 37-39]{LZ}, where an analogous
map is constructed in the setting of quasicompact separated schemes.
The key input is the following ``categorical gluing'' result, which is the analogue of \cite[Corollary 0.4]{LZ2} in our setup.
\begin{prop}\label{keygluing}
Let $\mathcal{V}\mathrm{sh}^{\mathrm{qcqs}}$ denote the category of small coproducts of qcqs
v-sheaves. Let $S\subset W$ be the
classes of strongly compactifiable and weakly compactifiable maps,
respectively. Let $P$ be the class
of proper weakly compactifiable maps, let $I$
be the class of separated local isomorphisms, and let $I_{\mathrm{qc}}\subset I$
be the class of quasicompact separated local isomorphisms. Then $P$
and $I$ are \emph{admissible subsets},
in the terminology of \cite{LZ2}, and every strongly compactifiable
morphism $f$ in $\mathcal{V}\mathrm{sh}^{\mathrm{qcqs}}$ can be factored as $f=p\circ q$ for
some $p\in P$ and $q\in I_{\mathrm{qc}}$.

In this notation, the natural map 
\[
\delta_{2}^{\ast}(\mathcal{V}\mathrm{sh}^{\mathrm{qcqs}})_{P,I}^{\mathrm{cart}}\to\mathcal{V}\mathrm{sh}^{\mathrm{qcqs}}_{W}
\]
is a categorical equivalence. Similarly, the natural map
\[
\delta_{3,\{3\}}^{\ast}(\mathcal{V}\mathrm{sh}^{\mathrm{qcqs}})_{P,I,\mathrm{all}}^{\mathrm{cart}}\to \delta_{2,\{2\}}^{\ast}(\mathcal{V}\mathrm{sh}^{\mathrm{qcqs}})_{W,\mathrm{all}}^{\mathrm{cart}}
\]
is a categorical equivalence. Finally, analogous claims hold with $\mathcal{V}\mathrm{sh}^{\mathrm{qcqs}}$ replaced by $\mathcal{V}\mathrm{sh}^{\mathrm{qcqs}}_{\amalg}$.
\end{prop}

\begin{proof}
All claims in the first paragraph follow immediately from the definitions. For the admissibility claims in particular, use the criterion in \cite[Remark 3.19]{LZ2}.
For the first categorical equivalence, argue exactly as in the proof of Corollary 0.4 in \cite{LZ2},
with the square
\[
\xymatrix{\delta_{3}^{\ast}(\mathcal{V}\mathrm{sh}^{\mathrm{qcqs}})_{P,I_{\mathrm{qc}},I}^{\mathrm{cart}}\ar[d]\ar[r] & \delta_{2}^{\ast}(\mathcal{V}\mathrm{sh}^{\mathrm{qcqs}})_{S,I}^{\mathrm{cart}}\ar[d]\\
\delta_{2}^{\ast}(\mathcal{V}\mathrm{sh}^{\mathrm{qcqs}})_{P,I}^{\mathrm{cart}}\ar[r] & \mathcal{V}\mathrm{sh}^{\mathrm{qcqs}}_{W}
}
\]
playing the role of the square in loc. cit. The argument for the second categorical equivalence is completely analogous. For the final claim, note that the passage from $\mathcal{V}\mathrm{sh}^{\mathrm{qcqs}}$ to $\mathcal{V}\mathrm{sh}^{\mathrm{qcqs}}_{\amalg}$ poses no additional complications, since by convention the various markings $?$ on $\mathcal{V}\mathrm{sh}^{\mathrm{qcqs}}$ induce markings on $\mathcal{V}\mathrm{sh}^{\mathrm{qcqs}}_{\amalg}$ by consideration of those edges which statically belong to ?.
\end{proof}

Returning to the problem at hand, consider the composite map
\begin{equation}
\delta_{3,\{1,2,3\}}^{\ast}(\mathcal{V}\mathrm{sh}^{\mathrm{qcqs}}_{\amalg})_{P,I,\mathrm{all}}^{\mathrm{cart}}\to(\mathcal{V}\mathrm{sh}^{\mathrm{qcqs}})^{\mathrm{op},\amalg}\overset{\mathrm{EO}^{\mathrm{I}}}{\longrightarrow}\mathcal{C}\mathrm{at}_{\infty}
\end{equation}
where the first arrow is the evident diagonal map, and $\mathrm{EO}^{\mathrm{I}}$
is the map encoding the association $X\mapsto\mathcal{D}_{\mathrm{\acute{e}t}}(X,\Lambda)$
together with its symmetric monoidal $\ast$-pullback functoriality.\footnote{This map exists by general nonsense, as in \cite[Section 2]{LZ} and the discussion on \cite[p. 133]{ECoD}.}

To construct the desired enhanced operation map, we will encode $!$-pushfowards by arguing as in \cite[pp. 37-39]{LZ}: in the source
of the map ??, we will pass to \emph{right} adjoints in direction
1, and \emph{left} adjoints in direction 2. 

To pass to right
adjoints in direction 1, we apply the dual of \cite[Proposition 1.4.4]{LZ}.
To apply this proposition, we need to check the relevant adjointability
in directions (1,2) and (1,3). The former is a special case of the
latter. For the latter, adjointability follows from proper base change
and the projection formula as in \cite[Lemma 3.2.5]{LZ}, both of which
hold in the present situation by Lemma \ref{swsanity}.(v). Therefore, passing to right adjoints in direction 1, we get a map
\[
\delta_{3,\{2,3\}}^{\ast}(\mathcal{V}\mathrm{sh}^{\mathrm{qcqs}}_{\amalg})_{P,I,\mathrm{all}}^{\mathrm{cart}}\to\mathcal{C}\mathrm{at}_{\infty}.
\]

By another application of \cite[Proposition 1.4.4]{LZ}, we now
pass to left adjoints in direction 2. Here the necessary adjointability
in direction (2,1) follows from the adjunction $f_{!}\vdash f^{\ast}$
for $f\in I$, and the adjointability in direction (2,3) follows from
\'etale base change and the projection formula for $j_{!}$ with $j\in I$,
which is trivial. Passing to left adjoints in direction 2, we get a map
\[
\mathrm{EO}':\delta_{3,\{3\}}^{\ast}(\mathcal{V}\mathrm{sh}^{\mathrm{qcqs}}_{\amalg})_{P,I,\mathrm{all}}^{\mathrm{cart}}\to\mathcal{C}\mathrm{at}_{\infty}.
\]

Finally, by the final claim in Proposition \ref{keygluing}, the map
\[
\delta_{3,\{3\}}^{\ast}(\mathcal{V}\mathrm{sh}^{\mathrm{qcqs}}_{\amalg})_{P,I,\mathrm{all}}^{\mathrm{cart}}\to\delta_{2,\{2\}}^{\ast}(\mathcal{V}\mathrm{sh}^{\mathrm{qcqs}}_{\amalg})_{W,\mathrm{all}}^{\mathrm{cart}}
\]
is a categorical equivalence of simplicial sets, so the induced functor
\[
f:\mathrm{Fun}\left(\delta_{2,\{2\}}^{\ast}(\mathcal{V}\mathrm{sh}^{\mathrm{qcqs}}_{\amalg})_{W,\mathrm{all}}^{\mathrm{cart}},\mathcal{C}\mathrm{at}_{\infty}\right)\to\mathrm{Fun}\left(\delta_{3,\{3\}}^{\ast}(\mathcal{V}\mathrm{sh}^{\mathrm{qcqs}}_{\amalg})_{P,I,\mathrm{all}}^{\mathrm{cart}},\mathcal{C}\mathrm{at}_{\infty}\right)
\]
is a categorical equivalence of $\infty$-categories by \cite[Proposition
1.2.7.3]{HTT}. Therefore, choosing any $x$ in the source of $f$ such
that $f(x)\simeq\mathrm{EO}'$, we obtain the desired map
\[
_{\mathcal{V}\mathrm{sh}^{\mathrm{qcqs}}}\mathrm{EO}:\delta_{2,\{2\}}^{\ast}(\mathcal{V}\mathrm{sh}^{\mathrm{qcqs}}_{\amalg})_{W,\mathrm{all}}^{\mathrm{cart}}\to\mathcal{C}\mathrm{at}_{\infty}.
\]
This completes the proof of Proposition \ref{seed}.

\subsection{Descent and codescent}

In this section we prove some descent and codescent properties satisfied
by the enhanced operation map constructed in the previous section.  These will be crucial inputs for the first iteration of the DESCENT algorithm.

Consider the enhanced operation map
\[
_{\mathcal{V}\mathrm{sh}^{\mathrm{qcqs}}}\mathrm{EO}:\delta_{2,\{2\}}^{\ast}(\mathcal{V}\mathrm{sh}^{\mathrm{qcqs}}_{\amalg})_{W,\mathrm{all}}^{\mathrm{cart}}\to\mathcal{C}\mathrm{at}_{\infty}
\]
constructed in the previous section. By restriction and passing to suitable
adjoints, we obtain from this the following functors (with notation as in Section 2.1):

1. A functor
\[
(_{\mathcal{V}\mathrm{sh}^{\mathrm{qcqs}}}\mathrm{EO^I})^{\otimes}:\left(\mathcal{V}\mathrm{sh}^{\mathrm{qcqs}}\right)^{\mathrm{op}}\to\mathrm{CAlg}(\mathcal{C}\mathrm{at}_{\infty})_{\mathrm{pr,st,cl}}^{\mathrm{L}}
\]
encoding the assignment $X\mapsto\mathcal{D}_{\mathrm{\acute{e}t}}(X,\Lambda)$
and all $\ast$-pullbacks together with their symmetric monoidal structures,
with all higher coherences.

2. A functor
\[
_{\mathcal{V}\mathrm{sh}^{\mathrm{qcqs}}}\mathrm{EO}_{!}:(\mathcal{V}\mathrm{sh}^{\mathrm{qcqs}})_{W}\to\mathcal{P}\mathrm{r}_{\mathrm{st}}^{\mathrm{L}}
\]
encoding the assignment $X\mapsto\mathcal{D}_{\mathrm{\acute{e}t}}(X,\Lambda)$
together with all $!$-pushforwards for weakly compactifiable maps, with all higher coherences.

2'. A functor
\[
_{\mathcal{V}\mathrm{sh}^{\mathrm{qcqs}}}\mathrm{EO}^{!}:(\mathcal{V}\mathrm{sh}^{\mathrm{qcqs}})_{W}^{\mathrm{op}}\to\mathcal{P}\mathrm{r}_{\mathrm{st}}^{\mathrm{R}}
\]
encoding the assignment $X\mapsto\mathcal{D}_{\mathrm{\acute{e}t}}(X,\Lambda)$
together with all $!$-pullbacks for weakly compactifiable maps, with all higher coherences.
\begin{prop}\label{keydescent}
Let $f:X\to Y$ be any surjective map in $\mathcal{V}\mathrm{sh}^{\mathrm{qcqs}}$.
\begin{enumerate}

\item $f$ is of universal $(_{\mathcal{V}\mathrm{sh}^{\mathrm{qcqs}}}\mathrm{EO^I})^{\otimes}$-descent.

\item If $f$ is weakly compactifiable, representable in locally spatial diamonds,
and cohomologically smooth, then $f$ is of universal $_{\mathcal{V}\mathrm{sh}^{\mathrm{qcqs}}}\mathrm{EO}_{!}$-codescent.
\end{enumerate}
\end{prop}
Here the terminology follows \cite[Definition 3.1.1]{LZ}.
\begin{proof}
Part 1. amounts to the claim that $\mathcal{D}_{\mathrm{\acute{e}t}}(-,\Lambda)$
is a v-sheaf of closed symmetric monoidal stable $\infty$-categories,
which is clear from its construction.

For 2., arguing as in Lemma 1.3.3 of \cite{GaiDG}, it is equivalent to prove that $f$ is of universal
$_{\mathcal{V}\mathrm{sh}^{\mathrm{qcqs}}}\mathrm{EO}^{!}$-descent.
In other words, if $f_{0}:X=X_{0}\to Y=X_{-1}$ is any map as in 2.,
with Cech nerve $f_{\bullet}:X_{\bullet}\to Y$, we need to prove
that
\[
\mathcal{D}_{\mathrm{\acute{e}t}}(Y,\Lambda)\simeq\lim_{n\in\boldsymbol{\Delta}}\mathcal{D}_{\mathrm{\acute{e}t}}(X_{n},\Lambda)
\]
where the transition maps are given by $!$-pullback. 

Quite generally, when computing the limit of a cosimplicial $\infty$-category,
it is equivalent to compute the limit of the associated semi-cosimplicial
$\infty$-category, by (the dual of) Lemma 6.5.3.7 and the subsequent
remarks in \cite{HTT}. Thus, let $\mathcal{C}^{\bullet}:N(\boldsymbol{\Delta}_{s})\to\mathcal{C}\mathrm{at}_{\infty}$
be the (augmented) semi-cosimplicial $\infty$-category with $\mathcal{C}^{n}=\mathcal{D}_{\mathrm{\acute{e}t}}(X_{n},\Lambda)$
and with the transition maps given by $!$-pullback. We need to prove
that 
\[
\mathcal{D}_{\mathrm{\acute{e}t}}(Y,\Lambda)\simeq\lim_{n\in\boldsymbol{\Delta}_{s}}\mathcal{C}^{n}.
\]
Let $\mathcal{D}^{\bullet}:N(\boldsymbol{\Delta}_{s})\to\mathcal{C}\mathrm{at}_{\infty}$
be the (augmented) semi-cosimplicial $\infty$-category with $\mathcal{D}^{n}=\mathcal{D}_{\mathrm{\acute{e}t}}(X_{n},\Lambda)$
and with the transition maps given by $\ast$-pullback. Then we have
a natural equivalence $\tau:\mathcal{D}^{\bullet}\overset{\sim}{\to}\mathcal{C}^{\bullet}$
sending $A_{n}\in\mathcal{D}_{\mathrm{\acute{e}t}}(X_{n},\Lambda)$
to $A_{n}\otimes Rf_{n}^{!}\Lambda$, and this equivalence is compatible
with the augmentations. Thus
\[
\mathcal{D}_{\mathrm{\acute{e}t}}(Y,\Lambda)\simeq\lim_{n\in\boldsymbol{\Delta}_{s}}\mathcal{D}^{n}\overset{\sim}{\to}\lim_{n\in\boldsymbol{\Delta}_{s}}\mathcal{C}^{n},
\]
where the first isomorphism follows from part 1 and the second isomorphism
is induced by $\tau$. This gives the desired result.
\end{proof}

\section{Running DESCENT}

\subsection{Decent v-stacks and fine morphisms}

Recall that in the introduction, we have defined decent v-stacks and
fine morphisms between them. In this section we study this notion
in detail.

To streamline the discussion, it will be convenient to make the following
definition.
\begin{defn}
A map of small v-stacks $f:X\to Y$ is \emph{representable in locally
separated locally spatial diamonds} if for all locally separated locally
spatial diamonds $T$ with a map $T\to Y$, the fiber product $X\times_{Y}T$
is a locally separated locally spatial diamond.
\end{defn}

Note that it is enough to quantify over all separated spatial
$T$. In this language, condition 1. in Definition \ref{decent} is exactly the condition that the diagonal be representable
in locally separated locally spatial diamonds.

\begin{warn}One must be slightly careful when using this definition. In particular, it is not clear to us whether $f:X\to Y$ being representable in locally separated locally spatial diamonds implies that $f$ is representable in locally spatial diamonds, although see Proposition \ref{lslsbasics}.i below for a partial result. Even if this were true, it is still not clear whether $f:X\to Y$ being representable in locally separated locally spatial diamonds is equivalent to $f$ being locally separated and representable in locally spatial diamonds. It seems plausible that such an equivalence is actually false, since the condition of $f$ being locally separated can be phrased as a quantification over open subsets of $|X|$, and in general $|X|$ can have very few open subsets.
\end{warn}

\begin{prop}\label{lslsbasics}
\emph{i. }Suppose $f:X\to Y$ is representable in locally separated locally
spatial diamonds. Then $f$ is representable in diamonds. If also $f$ is quasiseparated, then $f$ is representable in locally spatial diamonds.

\emph{ii. }The property of being representable in locally separated
locally spatial diamonds is stable under composition and base change.

\emph{iii. }Quasicompact injections are representable in locally separated
locally spatial diamonds.

\emph{iv. }If $f:X\to Y$ and $g:Y\to Z$ are maps of small v-stacks
such that $g\circ f$ is representable in locally separated locally
spatial diamonds, and $g$ is 0-truncated and quasiseparated, then
$f$ is representable in locally separated locally spatial diamonds.
\end{prop}

\begin{proof}
We show the first claim in i. Quite generally, to show that a map $f:X\to Y$ is representable in diamonds, it suffices to show that for any map $T\to Y$ with $T$ affinoid perfectoid, the fiber product $X\times_Y T$ is a diamond. Since affinoid perfectoids are separated spatial diamonds, the hypothesis in i. implies that $X \times_Y T$ is a (locally separated locally spatial) diamond, giving the first claim. The second claim in i. is left as a slightly tricky exercise for the reader (since we make no use of it in this paper).

Part ii. is clear, and iii. follows from \cite[Proposition 11.20]{ECoD}. For iv., let $T$ be any locally separated locally spatial diamond
with a map $T\to Y$. Then we have a cartesian diagram
\[
\xymatrix{X\times_{Y}T\ar[d]\ar[r] & X\times_{Z}T\ar[d]\\
Y\ar[r] & Y\times_{Z}Y
}
\]
of small v-stacks, and we need to see that $X\times_{Y}T$ is a locally
separated locally spatial diamond. Since $g\circ f$ is representable
in locally separated locally spatial diamonds, $X\times_{Z}T$ is
a locally separated locally spatial diamond. Then since $g$ is 0-truncated
and quasiseparated, $Y\to Y\times_{Z}Y$ is a quasicompact injection,
so $X\times_{Y}T\to X\times_{Z}T$ is a quasicompact injection. The
claim now follows from iii.
\end{proof}
\begin{prop}\label{decentartin}
Decent v-stacks are Artin v-stacks in the sense of \cite[Definition
IV.1.1]{FS}. In particular, the diagonal of any decent v-stack is quasiseparated
and representable in locally spatial diamonds.
\end{prop}

\begin{proof}
Suppose $X$ is decent, so $\Delta:X \to X \times X$ is representable in locally separated locally spatial diamonds. First we show that $\Delta$ is quasiseparated and representable
in locally spatial diamonds. Pick any chart $f:U\to X$. Then \cite[Remark
IV.1.4.ii]{FS} implies that $\Delta:X\to X\times X$ is quasiseparated,
and reduces us to checking that $\Delta$ is representable in diamonds.
Since $X$ is decent, this follows from the first part of Proposition \ref{lslsbasics}.i.

Refining any given chart for $X$ as in Remark \ref{decentremarks}.ii, we get a surjective
separated cohomologically smooth map $U'\to X$ from a locally spatial
diamond. Hence $X$ is Artin.
\end{proof}
\begin{prop}
If $X_{2}\to X_{1}\leftarrow X_{3}$ is any diagram of decent v-stacks,
then $X_{2}\times_{X_{1}}X_{3}$ is a decent v-stack.
\end{prop}

\begin{proof}
Pick charts $f_{i}:U_{i}\to X_{i}.$ Then $U_{2}\times_{X_{1}}U_{3}$
is a locally separated locally spatial diamond, by Remark \ref{decentremarks}.i applied
to $X_{1}$. One now verifies that $U_{2}\times_{X_{1}}U_{3}\to X_{2}\times_{X_{1}}X_{3}$
is a chart, by factoring it as the composition of maps
\[
U_{2}\times_{X_{1}}U_{3}\overset{g_{2}}{\to}X_{2}\times_{X_{1}}U_{3}\overset{g_{3}}{\to}X_{2}\times_{X_{1}}X_{3}
\]
where $g_{i}$ is a base change of $f_{i}$.

For the condition on the diagonal, let $u:X_{2}\times_{X_{1}}X_{3}\to X_{2}\times X_{3}$
be the pullback of $\Delta_{X_{1}}:X_{1}\to X_{1}\times X_{1}$ along
$X_{2}\times X_{3}\to X_{1}\times X_{1}$, so $u$ is representable in locally separated locally spatial diamonds by the decency of $X_1$ and quasiseparated by Proposition \ref{decentartin}. Now
consider the commutative diagram
\[
\xymatrix{X_{2}\times_{X_{1}}X_{3}\ar[d]^{u}\ar[rr]^{\Delta_{X_{2}\times_{X_{1}}X_{3}}\;\;\;\;\;\;\;\;\;\;} &  & (X_{2}\times_{X_{1}}X_{3})\times(X_{2}\times_{X_{1}}X_{3})\ar[d]^{u\times u}\\
X_{2}\times X_{3}\ar[rr]^{\Delta_{X_{2}\times X_{3}}\;\;\;\;\;\;} &  & (X_{2}\times X_{3})\times(X_{2}\times X_{3})
}
\]
of small v-stacks. Using the decency of $X_{2}$ and $X_{3}$, one
sees by repeated applications of Proposition \ref{lslsbasics}.ii that $\Delta_{X_{2}\times X_{3}}$
and then also $\Delta_{X_{2}\times X_{3}}\circ u$ are both representable
in locally separated locally spatial diamonds. Now going around the
diagram via the upper right, we also get that 
\[
(u\times u)\circ\Delta_{X_{2}\times_{X_{1}}X_{3}}\simeq\Delta_{X_{2}\times X_{3}}\circ u
\]
is representable in locally separated locally spatial diamonds. Since
$u$ is 0-truncated and quasiseparated, it is formal that $u\times u$ is 0-truncated and quasiseparated, by factoring it as a composition of pullbacks of $u$ and using
Proposition \ref{lslsbasics}.ii again.
Therefore, applying Proposition \ref{lslsbasics}.iv with $f=\Delta_{X_{2}\times_{X_{1}}X_{3}}$
and $g=u\times u$, we deduce that $\Delta_{X_{2}\times_{X_{1}}X_{3}}$
is representable in locally separated locally spatial diamonds, as
desired.
\end{proof}

We now turn to the study of fine morphisms. We begin with some remarks on the definition.

\begin{rem}\label{fineremark}
Let $f$ be a fine morphism, with
\[
\xymatrix{W\ar[d]^{b}\ar[r]^{g} & V\ar[d]^{a}\\
X\ar[r]^{f} & Y
}
\]
a commutative diagram witnessing the fineness of $f$. Successively refining $V$ and $W$ as in Remark \ref{decentremarks}.ii, we can extend this to a commutative diagram
\[
\xymatrix{W'\ar[d]^{\beta}\ar[r]^{g'} & V'\ar[d]^{\alpha}\\
W\ar[d]^{b}\ar[r]^{g} & V\ar[d]^{a}\\
X\ar[r]^{f} & Y
}
\]
where $V'$ and $W'$ are separated locally spatial diamonds, $\alpha$ and $\beta$ are surjective separated local isomorphisms, and $a \circ \alpha$ and $b \circ \beta$ are clean charts. Using \cite[Proposition 22.3]{ECoD}, the condition that $g$ is locally on $W$ compactifiable of finite dim.trg is then \emph{equivalent} to the analogous condition for $g'$. However, $g'$ is a map between separated v-sheaves, thus is separated itself, so using \cite[Proposition 22.3]{ECoD} again, one sees that the condition on $g'$ boils down to the condition that $g'$ is compactifiable of locally finite dim.trg. Summarizing this discussion, we conclude that a given morphism $f$ is fine if and only if there is a commutative diagram
\[
\xymatrix{W'\ar[d]^{b'}\ar[r]^{g'} & V'\ar[d]^{a'}\\
X\ar[r]^{f} & Y
}
\]
where $a'$ and $b'$ are clean charts and $g'$ is compactifiable of locally finite dim.trg. 

Similar arguments show that if $f:X \to Y$ is any map of decent v-stacks, then $f$ is $\ell$-cohomologically smooth in the sense of Definition \ref{fine}.ii if and only if it is $\ell$-cohomologically smooth in the sense of \cite[Definition IV.1.11]{FS}.
\end{rem}

\begin{rem}
In the definition of a fine morphism, one might instead ask for an a priori
stronger condition, namely the existence of a commutative diagram
\[
\xymatrix{W'\ar[d]^{b'}\ar[dr]^{g'}\\
U\ar[d]^{\tilde{a}}\ar[r] & V\ar[d]^{a}\\
X\ar[r]^{f} & Y
}
\]
where $a$ is a chart for $Y$, the square is cartesian, $b'$ is
a chart for $U$, and $g'$ is locally on $W'$ compactifiable of
finite dim.trg. The existence of such a diagram certainly implies
that $f$ is fine, since $\tilde{a}\circ b':W'\to X$ is a chart
for $X$. In fact, these two conditions are equivalent.

To see this, pick a commutative square as in the definition of a fine morphism.
Set $U=X\times_{Y}V$, so we get a commutative diagram
\[
\xymatrix{W\ar[r]^{i}\ar[dr]^{b} & U\ar[r]^{h}\ar[d]^{\tilde{a}} & V\ar[d]^{a}\\
 & X\ar[r]^{f} & Y
}
\]
where $a$ and $b$ are charts, the square is cartesian, and $g=h\circ i$.
Then setting $W'=W\times_{b,X,\tilde{a}}U$, this diagram extends
to a diagram
\[
\xymatrix{W'\ar[d]^{\tilde{\tilde{a}}}\ar[dr]^{b'}\\
W\ar[r]^{i}\ar[dr]^{b} & U\ar[r]^{h}\ar[d]^{\tilde{a}} & V\ar[d]^{a}\\
 & X\ar[r]^{f} & Y
}
\]
where the trapezoid is cartesian. (\textbf{Warning!} The triangle
spanned by $\tilde{\tilde{a}},b',i$ is typically not commutative.)
Using the alternative presentation $W'=W\times_{Y}V$, decency of
$Y$ implies by Remark \ref{decentremarks}.i that $W'$ is a locally separated locally
spatial diamond. One then sees (using that $b$ is a chart) that $b'$
is a chart. 

To conclude, it is enough to see that $g':=h\circ b'$ is locally
on $W'$ compactifiable of finite dim.trg. For this, observe that
by our assumptions on $g=h\circ i$ and $a$, $a\circ g=a\circ h\circ i=f\circ b$
is representable in locally spatial diamonds and is locally on $W$
compactifiable of finite dim.trg. Indeed, the claimed properties are
true for $g$ and $a$ separately by assumption; to get it for the
composition, refine the chart $a$ as in Remark \ref{decentremarks}.ii if necessary.
Then $h\circ b'$ is the base change of $f\circ b$ along $a$, which
gives the desired conclusion.
\end{rem}

\begin{prop}
\emph{o. }If $f:X\to Y$ is fine, then for \emph{any }commutative
diagram
\[
\xymatrix{W\ar[d]\ar[r] & V\ar[d]\\
X\ar[r] & Y
}
\]
where the vertical maps are charts, the map $W\to V$ is locally on
$W$ compactifiable of finite dim.trg.

\emph{i. }Fine morphisms are stable under composition.

\emph{ii. }If 
\[
\xymatrix{\tilde{X}\ar[r]^{\tilde{f}}\ar[d] & \tilde{Y}\ar[d]\\
X\ar[r]^{f} & Y
}
\]
is any Cartesian diagram of decent v-stacks, and $f$ is fine, then
also $\tilde{f}$ is fine.

\emph{iii. }If 
\[
\xymatrix{ & Y\ar[dr]^{g}\\
X\ar[rr]^{h}\ar[ur]^{f} &  & Z
}
\]
is any commutative diagram of decent v-stacks, and $h$ and $g$ are
fine, then so is $f$.

\emph{iv. }If $f:X\to Y$ is a map of decent v-stacks which is separated
and representable in locally spatial diamonds, then $f$ is fine
if and only if $f$ is compactifiable of locally finite dim.trg.
\end{prop}

\begin{proof}
o. By Remark \ref{fineremark}, it suffices to show that if 
$f \colon X \to Y$ is a fine morphism,
there are
commutative
squares
\[
\vcenter{\xymatrix{W\ar[d]^b \ar[r]^g & V\ar[d]^a \\
X\ar[r]^f & Y
}} \text{ and }
\vcenter{\xymatrix{W' \ar[d]^{b'} \ar[r]^{g'} & V'\ar[d]^{a'} \\
X\ar[r]^f & Y
}}\,,
\]
where the vertical maps are clean charts, and $g'$ is compactifiable
of locally finite dim.trg, then
$g$ is also compactifiable of locally finite dim.trg.

We apply various parts of \cite[Proposition 22.3]{ECoD} to
show that the following maps are compactifiable:
$W \times_X W' \to W'$ by (ii),
$W \times_X W' \to V \times_Y W'$ by (viii),
$W \times_X W' \to V$ by (iv),
and finally, $g \colon W \to V$ by (vii).

Let $w \in W$, let $w' \in W'$ satisfy $b'(w')=b(w)$, and choose
open neighborhoods
$U \owns w$, $U' \owns w'$ so that $b|_U$ and $(a'\circ g')|_{U'}=(f \circ b')|_{U'}$ are of finite dim.trg.
Then, using a chain of maps similar to the one
in the previous paragraph, we see that $U \times_X U' \to V$ is of finite dim.trg.  By \cite[Proposition 23.11]{ECoD},
$a'$ is universally open, so $a^{-1}(a'(U')) \subset W$ is open.
Then the restriction of $g$ to $U \cap a^{-1}(a'(U'))$ has finite dim.trg.
Since $w$ was chosen arbitrarily, $g$ is of locally finite dim.trg.


i. Suppose $X_{1}\to X_{2}$ and $X_{2}\to X_{3}$ are fine maps
of decent v-stacks. Pick commutative squares
\[
\xymatrix{U\ar[d]\ar[r] & V\ar[d]\\
X_{1}\ar[r] & X_{2}
}
\]
and 
\[
\xymatrix{V'\ar[d]\ar[r] & W\ar[d]\\
X_{2}\ar[r] & X_{3}
}
\]
witnessing the fineness of these maps. Set $T=V\times_{X_{2}}V'$.
Then we get a commutative diagram
\[
\xymatrix{ & U\times_{V}T\ar[dl]\ar[r] & T\ar[dr]\ar[dl]\\
U\ar[dr]\ar[r] & V\ar[dr] &  & V'\ar[dl]\ar[r] & W\ar[dl]\\
 & X_{1}\ar[r] & X_{2}\ar[r] & X_{3}
}
\]
and throwing away most of it we get a diagram
\[
\xymatrix{U\times_{V}T\ar[r]\ar[d] & W\ar[d]\\
X_{1}\ar[r] & X_{3}
}
\]
witnessing the fineness of $X_{1}\to X_{3}$.

ii. Pick a diagram 
\[
\xymatrix{W\ar[d]^{b}\ar[r]^{g} & V\ar[d]^{a}\\
X\ar[r]^{f} & Y
}
\]
witnessing the fineness of $f$. Let $U\to\tilde{Y}$ be a chart for
$\tilde{Y}$. Then we get a commutative diagram
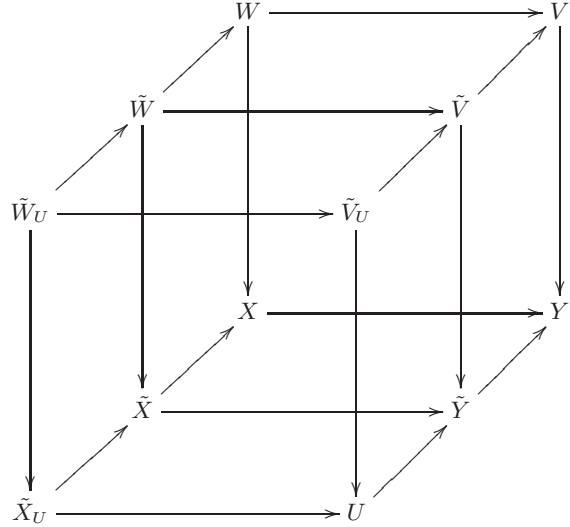
\[
\xymatrix{ &  & W\ar[rrr]\ar[ddd] &  &  & V\ar[ddd]\\
 & \tilde{W}\ar[rrr]\ar[ddd]\ar[ur] &  &  & \tilde{V}\ar[ddd]\ar[ur]\\
\tilde{W}_{U}\ar[rrr]\ar[ddd]\ar[ur] &  &  & \tilde{V}_{U}\ar[ddd]\ar[ur]\\
 &  & X\ar[rrr] &  &  & Y\\
 & \tilde{X}\ar[rrr]\ar[ur] &  &  & \tilde{Y}\ar[ur]\\
\tilde{X}_{U}\ar[rrr]\ar[ur] &  &  & U\ar[ur]
}
\]
where all vertical arrows are locally separated smooth etc. and the
parallelograms involving slanted arrows are all cartesian. Pick a
chart $T\to\tilde{X}_{U}$. Then we get a commutative square
\[
\xymatrix{T\times_{\tilde{X}_{U}}\tilde{W}_{U}\ar[d]\ar[r] & \tilde{V}_{U}\ar[d]\\
\tilde{X}\ar[r] & \tilde{Y}
}
\]
which I claim witnesses the fineness of $\tilde{X}\to\tilde{Y}$.

iii. We can find commutative diagrams
\[
\vcenter{\xymatrix{V\ar[d] \ar[r] & U_1\ar[d] \\
Y\ar[r]^g & Z
}} \text{ and }
\vcenter{\xymatrix{W \ar[d] \ar[r] & U_2 \ar[d] \\
X\ar[r]^h & Z
}}\,,
\]
where the vertical arrows are clean charts and the horizontal
arrows
are compactifiable of locally finite dim.trg.
Then
\[ \xymatrix{V \times_Y W \ar[r]\ar[d] & V \times_Z U_2 \ar[r]\ar[d] & U_1 \times_Z U_2 \ar[d] \\ X \ar[r]^f & Y \ar[r]^g & Z } \]
is also a commutative diagram where the vertical arrows are clean charts,
and $V \times_Y W \to U_1 \times_Z U_2$ and
$V \times_Z U_2 \to U_1 \times_Z U_2$ are compactifiable of locally finite
dim.trg.~by o.  Then $V \times_Y W \to V \times_Z U_2$
is compactifiable by \cite[Proposition 22.3(viii)]{ECoD}, and it has
locally finite dim.trg.


iv. Choose a commutative diagram
\[
\xymatrix{W\ar[d]^{b}\ar[r]^{g} & V\ar[d]^{a}\\
X\ar[r]^{f} & Y
}
\]
such that the vertical arrows are clean charts.  By o., $f$ is fine if
and only if $g$ is compactifiable of locally finite dim.trg.  By
\cite[Proposition 22.3(iv,vii,viii)]{ECoD}, $g$ is compactifiable iff $f$ is
compactifable.  By an argument similar to that of the proof of o.,
$b$ is universally open, and $g$ is locally on $W$ of finite dim.trg.~iff $f$
is locally on $X$ of finite dim.trg.
\end{proof}

\subsection{A fragment of DESCENT}

In this section we describe the fragment of Liu-Zheng's DESCENT algorithm
we will need. The strange numberings below are chosen to match the
numbering in \cite{LZ}.

\textbf{Input 0.} Suppose given a 4-marked $\infty$-category $(\tilde{\mathcal{C}},\tilde{\mathcal{E}}_{s},\tilde{\mathcal{E}}',\tilde{\mathcal{E}}'',\tilde{\mathcal{F}})$
together with a full subcategory $\mathcal{C}\subset\tilde{\mathcal{C}}$.
Write $\mathcal{E}_{s}=\tilde{\mathcal{E}}_{s}\cap\mathcal{C}$, and
likewise $\mathcal{E}'$, $\mathcal{E}''$, $\mathcal{F}$. They are
assumed to satisfy:

1. $\tilde{\mathcal{C}}$ is geometric, $\mathcal{C}\subseteq\tilde{\mathcal{C}}$
is stable under finite limits, and for all small coproducts $X=\coprod_{i}X_{i}$
in $\tilde{\mathcal{C}}$, $X$ belongs to $\mathcal{C}$ if and only
if all $X_{i}$ belong to $\mathcal{C}$.

3. $\tilde{\mathcal{E}}_{s},\tilde{\mathcal{E}}',\tilde{\mathcal{E}}'',\tilde{\mathcal{F}}$
are stable under composition, pullback and small coproducts, and $\tilde{\mathcal{E}}'\subseteq\tilde{\mathcal{E}}''\subseteq\tilde{\mathcal{F}}$.

4. For every object $X$ of $\tilde{\mathcal{C}}$, there exists an
edge $f:Y\to X$ in $\tilde{\mathcal{E}}_{s}\cap\tilde{\mathcal{E}}'$
with $Y$ in $\mathcal{C}$. Such an edge is called an \emph{atlas
}for $X$.

5. For every pullback square
\[
\xymatrix{W\ar[d]\ar[r] & Z\ar[d]\\
Y\ar[r]^{f} & X
}
\]
with $X\in\mathcal{\tilde{C}}$, $Y\in\mathcal{C}$ and $Z\in\mathcal{C}$,
and $f$ an atlas, then also $W\in\mathcal{C}$. Intuitively, ``atlas
maps are representable in $\mathcal{C}$''.

\textbf{Input I.} Suppose given an enhanced operation map \[
_{\mathcal{C}}\mathrm{EO}:\delta_{2,\{2\}}^{\ast}(\mathcal{C}_{\amalg})_{\mathcal{F},\mathrm{all}}^{\mathrm{cart}}\to\mathcal{C}\mathrm{at}_{\infty}
\]
satisfying the following properties.

\textbf{P0-P2.}\footnote{This is of course redundant, since it is exactly condition 1. in the definition of an enhanced operation map. We have written it out nonetheless to aid the reader in comparing with \cite[Section 4.1]{LZ}, where the notion of an enhanced operation map is not explicitly formalized and so P0-P2 have actual content.} The functor
\[
_{\mathcal{C}}\mathrm{EO}^{\mathrm{I}}:\mathcal{C}^{\mathrm{op},\amalg}\to\mathcal{C}\mathrm{at}_{\infty}
\]
induced by restriction to the ``all'' direction is a weak Cartesian
structure, and the induced functor $({_{\mathcal{C}}\mathrm{EO}^{\mathrm{I}}})^{\otimes}$
factors through $\mathrm{CAlg}(\mathcal{C}\mathrm{at}_{\infty})_{\mathrm{pr,st,cl}}^{\mathrm{L}}$
and sends small coproducts to products.

As in section 2.1, from $_{\mathcal{C}}\mathrm{EO}$ we obtain a map
\[
_{\mathcal{C}}\mathrm{EO}_{!}^{\ast}:\delta_{2,\{2\}}^{\ast}\mathcal{C}{}_{\mathcal{F},\mathrm{all}}^{\mathrm{cart}}\to\mathcal{P}\mathrm{r}_{\mathrm{st}}^{\mathrm{L}},
\]
which restricts further to maps
\[
_{\mathcal{C}}\mathrm{EO}^{\ast}:\mathcal{C}^{\mathrm{op}}\to\mathcal{P}\mathrm{r}_{\mathrm{st}}^{\mathrm{L}}\;\;\;\;\;\mathrm{and}\;\;\;\;\;{}_{\mathcal{C}}\mathrm{EO}_{!}:\mathcal{C}_{\mathcal{F}}\to\mathcal{P}\mathrm{r}_{\mathrm{st}}^{\mathrm{L}}.
\]
As before we write $\mathcal{D}(X)$ for the image of a 0-cell $X\in\mathcal{C}$
under either of these maps, and $f^{\ast}$ resp. $f_{!}$ for the
image of a 1-cell $f:Y\to X$ under the first resp. second map. Now
we can state the remaining properties we impose.

\textbf{P3.} If $f:Y\to X$ is an edge in $\mathcal{E}_{s}$, then
$f^{\ast}:\mathcal{D}(X)\to\mathcal{D}(Y)$ is conservative.

\textbf{P4.} If $f$ is an edge in $\mathcal{E}_{s}\cap\mathcal{E}''$,
then $f$ is of universal $_{\mathcal{C}}\mathrm{EO}^{\otimes}$-descent
and of universal $_{\mathcal{C}}\mathrm{EO}_{!}$-codescent.

\textbf{P5.} Let
\[
\xymatrix{W\ar[r]^{g}\ar[d]^{q} & Z\ar[d]^{p}\\
Y\ar[r]^{f} & X
}
\]
be a cartesian diagram in $\mathcal{C}$, with $f$ in $\mathcal{E}'$.
Then:

1) The square
\[
\xymatrix{\mathcal{D}(Z)\ar[r]^{p^{\ast}}\ar[d]^{g^{\ast}} & \mathcal{D}(X)\ar[d]^{f^{\ast}}\\
\mathcal{D}(W)\ar[r]^{q^{\ast}} & \mathcal{D}(Y)
}
\]
is right-adjointable, with a right adjoint a square in $\mathcal{P}\mathrm{r}_{\mathrm{st}}^{\mathrm{R}}$.

2) If $p$ is also in $\mathcal{E}'$, then the square
\[
\xymatrix{\mathcal{D}(X)\ar[r]^{f_{!}}\ar[d]^{p^{\ast}} & \mathcal{D}(Y)\ar[d]^{q^{\ast}}\\
\mathcal{D}(Z)\ar[r]^{g_{!}} & \mathcal{D}(W)
}
\]
is right-adjointable.

\textbf{P5bis.} Same as P5 but with $\mathcal{E}'$ replaced by $\mathcal{E}''$.

\textbf{Output I.} This consists of an enhanced operation map
\[
_{\mathcal{\tilde{C}}}\mathrm{EO}:\delta_{2,\{2\}}^{\ast}(\mathcal{\tilde{C}}_{\amalg})_{\tilde{\mathcal{F}},\mathrm{all}}^{\mathrm{cart}}\to\mathcal{C}\mathrm{at}_{\infty}
\]
extending Input I, and satisfying the obvious analogues of \textbf{P0}-\textbf{P5bis}
with $\tilde{\mathcal{C}},\tilde{\mathcal{E}}_{s},\tilde{\mathcal{E}}',\tilde{\mathcal{E}}'',\tilde{\mathcal{F}}$
in place of $\mathcal{C},\mathcal{E}_{s},\mathcal{E}',\mathcal{E}'',\mathcal{F}$.

We now have the following key theorem, which is the fragment of DESCENT
we will need.
\begin{thm}\label{descent}
Fix an Input 0. Then every Input I can be extended to an Output I
in an essentially unique way.
\end{thm}

\begin{proof}
This is a special case of (the proof of) \cite[Theorem 4.1.8.(1)]{LZ}.
\end{proof}

\subsection{First iteration of DESCENT}

We begin by running DESCENT with the following inputs.

For Input 0, we make the following choices.
\begin{itemize}
\item $\tilde{\mathcal{C}}=\mathcal{D}\mathrm{ia}^{\mathrm{qs.lsep.lspat}}$
where $\mathcal{D}\mathrm{ia}^{\mathrm{qs.lsep.lspat}}$ is the category
of quasiseparated locally separated locally spatial diamonds.
\item $\mathcal{C}=\mathcal{D}\mathrm{ia}^{\mathrm{sep.spat}}$ where
$\mathcal{D}\mathrm{ia}^{\mathrm{sep.spat}}\subset\mathcal{D}\mathrm{ia}^{\mathrm{qs.lsep.lspat}}$
is the full subcategory spanned by small coproducts of separated spatial diamonds.
\item $\tilde{\mathcal{E}}_{s}$ is the set of morphisms which are surjective as maps of v-sheaves.
\item $\tilde{\mathcal{E}}'$ is the set of (locally separated) \'etale morphisms
in $\mathcal{D}\mathrm{ia}^{\mathrm{qs.lsep.lspat}}$.
\item $\tilde{\mathcal{E}}''$ is the set of (locally separated) cohomologically
smooth morphisms in $\mathcal{D}\mathrm{ia}^{\mathrm{qs.lsep.lspat}}$.
\item $\tilde{\mathcal{F}}$ is the set of fine morphisms.
\end{itemize}
It is easy to see that these choices satisfy the conditions required
of an Input 0. The only point which isn't immediate is condition 5.,
which follows from the next lemma.
\begin{lem}
Let 
\[
\xymatrix{W\ar[r]\ar[d] & Y\ar[d]^{g}\\
X\ar[r]^{f} & Z
}
\]
be a cartesian diagram of v-sheaves, where $X$ and $Y$ are separated
spatial diamonds, and $Z$ is a quasiseparated locally separated locally
spatial diamond. Then $W$ is a separated spatial diamond.
\end{lem}

\begin{proof}
Since $W=X\times_{Z}Y$ is a subfunctor of $X\times Y$, and $X\times Y$
is separated, we immediately get that $W$ is separated. Local spatiality
of $W$ is clear, so it remains to see that $W$ is qcqs. Since $X$
and $Y$ are quasicompact and $Z$ is quasiseparated, the maps $f,g$
are quasicompact. In particular, $|f|$ and $|g|$ have quasicompact
images in $|Z|$. By local spatiality of $Z$, we may pick some sufficiently
large quasicompact open subfunctor $U\subset Z$ such that $|U|\supset(\mathrm{im}|f|\cup\mathrm{im}|g|)$.
Then since $Z$ is quasiseparated, $U$ is automatically qcqs, so
$W=X\times_{Z}Y=X\times_{U}Y$ is a fiber product of qcqs objects,
and therefore is qcqs as desired.
\end{proof}
As Input I, we take the enhanced operation map constructed in section
3.2 for a specific choice of coefficient ring $\Lambda$, restricted from $\mathcal{V}\mathrm{sh}^{\mathrm{qcqs}}$ to $\mathcal{D}\mathrm{ia}^{\mathrm{sep.spat}}$,
noting that in $\mathcal{D}\mathrm{ia}^{\mathrm{sep.spat}}$ weakly
compactifiable maps exactly coincide with fine maps. P0-P2 and P3
are clear. P4 follows from Proposition \ref{keydescent}. P5 and P5bis follow
from a combination of smooth and proper base change. 

Theorem \ref{descent} now applies, yielding an enhanced operation map on $\tilde{\mathcal{C}}$,
and in particular a functor
\[
_{\mathcal{D}\mathrm{ia}^{\mathrm{qs.lsep.lspat}}}\mathrm{EO}:\delta_{2,\{2\}}^{\ast}(\mathcal{D}\mathrm{ia}^{\mathrm{qs.lsep.lspat}}_{\amalg})_{\mathcal{F},\mathrm{all}}^{\mathrm{cart}}\to\mathcal{C}\mathrm{at}_{\infty}
\]
where $\mathcal{F}$ denotes the set of fine morphisms.

\subsection{Second iteration of DESCENT}

We run DESCENT with the following input.

For Input 0, we make the following choices.
\begin{itemize}
\item $\tilde{\mathcal{C}}=\mathcal{D}\mathrm{ia}^{\mathrm{lsep.lspat}}$
where $\mathcal{D}\mathrm{ia}^{\mathrm{lsep.lspat}}$ is the category
of locally separated locally spatial diamonds.
\item $\mathcal{C}=\mathcal{D}\mathrm{ia}^{\mathrm{qs.lsep.lspat}}$
where $\mathcal{D}\mathrm{ia}^{\mathrm{qs.lsep.lspat}}\subset\mathcal{D}\mathrm{ia}^{\mathrm{lsep.lspat}}$
is the category of quasiseparated locally separated locally spatial
diamonds.
\item $\tilde{\mathcal{E}}_{s}$ is the set of surjective morphisms of v-sheaves.
\item $\tilde{\mathcal{E}}'$ is the set of (locally separated) \'etale morphisms
in $\mathcal{D}\mathrm{ia}^{\mathrm{lsep.lspat}}$.
\item $\tilde{\mathcal{E}}''$ is the set of (locally separated) cohomologically
smooth morphisms in $\mathcal{D}\mathrm{ia}^{\mathrm{lsep.lspat}}$.
\item $\tilde{\mathcal{F}}$ is the set of fine morphisms in $\mathcal{D}\mathrm{ia}^{\mathrm{lsep.lspat}}$.
\end{itemize}
It is easy to see that these choices satisfy the conditions required
of an Input 0. The only point which isn't immediate is condition 5.,
which follows from the next lemma.
\begin{lem}
Let 
\[
\xymatrix{W\ar[r]\ar[d] & Y\ar[d]^{g}\\
X\ar[r]^{f} & Z
}
\]
be a cartesian diagram of small v-sheaves, where $X$ and $Y$ are
quasiseparated locally separated locally spatial diamonds, and $Z$
is a locally separated locally spatial diamond. Then $W$ is a quasiseparated
locally separated locally spatial diamond.
\end{lem}

\begin{proof}
It is clear that $W$ is a locally separated locally spatial diamond.
We need to see that $W$ is quasiseparated. Since $W=X\times_{Z}Y\subset X\times Y$
is a subfunctor of $X\times Y$, and quasiseparatedness passes to
subfunctors, it's enough to see that $X\times Y$ is quasiseparated.
But small v-sheaves form an algebraic topos by \cite[Proposition
8.3]{ECoD}, so for any quasiseparated objects $X$ and $Y$, also $X\times Y$
is quasiseparated by \cite[VI, Proposition 2.2.(ii)]{SGA4T2}. 
\end{proof}
As Input I, we take the enhanced operation map constructed as Output
I of the first iteration. P0-P5bis are automatic, since they hold for
any Output I. 

Theorem \ref{descent} now applies again, yielding an enhanced operation map
on $\tilde{\mathcal{C}}$, and in particular a functor
\[
_{\mathcal{D}\mathrm{ia}^{\mathrm{lsep.lspat}}}\mathrm{EO}:\delta_{2,\{2\}}^{\ast}(\mathcal{D}\mathrm{ia}^{\mathrm{lsep.lspat}}_{\amalg})_{\mathcal{F},\mathrm{all}}^{\mathrm{cart}}\to\mathcal{C}\mathrm{at}_{\infty}
\]
where $\mathcal{F}$ denotes the set of fine morphisms.

\subsection{Third iteration of DESCENT}

We run DESCENT with the following input.

For Input 0, we make the following choices.
\begin{itemize}
\item $\tilde{\mathcal{C}}=\mathcal{V}\mathrm{stk}^{\mathrm{dct}}$ where
$\mathcal{V}\mathrm{stk}^{\mathrm{dct}}$ is the category of decent
v-stacks.
\item $\mathcal{C}=\mathcal{D}\mathrm{ia}^{\mathrm{lsep.lspat}}$ where
$\mathcal{D}\mathrm{ia}^{\mathrm{lsep.lspat}}\subset\mathcal{V}\mathrm{stk}^{\mathrm{dct}}$
is the category of locally separated locally spatial diamonds.
\item $\tilde{\mathcal{E}}_{s}$ is the set of surjective morphisms of v-stacks.
\item $\tilde{\mathcal{E}}'=\tilde{\mathcal{E}}''$ is the set of cohomologically
smooth morphisms in $\mathcal{V}\mathrm{stk}^{\mathrm{dct}}$.
\item $\tilde{\mathcal{F}}$ is the set of fine morphisms in $\mathcal{V}\mathrm{stk}^{\mathrm{dct}}$.
\end{itemize}
It is easy to see that these choices satisfy the conditions required
of an Input 0. As Input I, we take the enhanced operation map constructed
as Output I of the second iteration. P0-P5bis are automatic, since they
hold for any Output I. 

Theorem \ref{descent} now applies again, yielding an enhanced operation map
on $\tilde{\mathcal{C}}$, and in particular a functor
\[
_{\mathcal{V}\mathrm{stk}^{\mathrm{dct}}}\mathrm{EO}:\delta_{2,\{2\}}^{\ast}(\mathcal{V}\mathrm{stk}_{\amalg}^{\mathrm{dct}})_{\mathcal{F},\mathrm{all}}^{\mathrm{cart}}\to\mathcal{C}\mathrm{at}_{\infty}
\]
where $\mathcal{F}$ denotes the set of fine morphisms.

\subsection{Endgame}
In this section we complete the proof of Theorem \ref{mainthmtechnical}.

Let 
\begin{equation} _{\mathcal{V}\mathrm{stk}^{\mathrm{dct}}}\mathrm{EO}:\delta_{2,\{2\}}^{\ast}(\mathcal{V}\mathrm{stk}_{\amalg}^{\mathrm{dct}})_{\mathcal{F},\mathrm{all}}^{\mathrm{cart}}\to\mathcal{C}\mathrm{at}_{\infty}
\end{equation}
be the enhanced operation map obtained as the output of the third iteration of DESCENT in the previous section. We freely use the notation from section 2.1 for various subordinate structures obtained from this map.

\begin{prop}\label{compatibleone} For any $X \in \mathcal{V}\mathrm{stk}^{\mathrm{dct}}$, we have $\mathcal{D}(X) \cong \mathcal{D}_{\mathrm{\acute{e}t}}(X,\Lambda)$ compatibly with the symmetric monoidal structures and with all $\ast$-pullbacks.
\end{prop}

\begin{proof}For this we need to dig into the proof of Theorem \ref{descent}. For the moment, let $\mathcal{C} \subset \tilde{\mathcal{C}}$ be as in the general setup for DESCENT. The proof of Theorem \ref{descent} extends the association $X \rightsquigarrow \mathcal{D}(X)$ from $\mathcal{C}$ to $\tilde{\mathcal{C}}$ as follows. For any $X \in \tilde{\mathcal{C}}$, let $X_0 \to X$ be a choice of atlas, with Cech nerve $X_\bullet \to X$. By the general properties of any Input 0, we have $X_n \in \mathcal{C}$ for all $n \geq 0$. Then $\mathcal{D}(X)$ with its symmetric monoidal structure is constructed as the limit of the cosimplicial $\infty$-category $n\in \boldsymbol{\Delta} \mapsto \mathcal{D}(X_n)$, where the transition maps are given by $\ast$-pullback. (Of course, the proof of Theorem \ref{descent} also accounts for the ambiguity arising from the choice of a particular atlas, roughly by taking a further limit over all possible atlases, and also handles the $\ast$-pullbacks.)

Returning to the situation at hand, we already know that $\mathcal{D}_{\mathrm{\acute{e}t}}(-,\Lambda)$ is a v-sheaf of symmetric monoidal $\infty$-categories on the category of \emph{all} small v-stacks. Moreover, the enhanced operation map constructed in Theorem \ref{seed} satisfies $\mathcal{D}(X) \cong \mathcal{D}_{\mathrm{\acute{e}t}}(X,\Lambda)$ for all $X \in \mathcal{D}\mathrm{ia}^{\mathrm{sep.spat}}$ by design, compatibly with the symmetric monoidal structure and with $\ast$-pullbacks. Since in each of our three iterations of DESCENT the atlas maps were chosen to be v-covers, it is now clear by the argument in the previous paragraph that the enhanced operation map obtained at the output of each successive iteration still satisfies $_{\tilde{\mathcal{C}}} \mathrm{EO}(X) \cong \mathcal{D}_{\mathrm{\acute{e}t}}(X,\Lambda)$ (for $\tilde{\mathcal{C}} \in \{\mathcal{D}\mathrm{ia}^{\mathrm{qs.lsep.lspat}}, \mathcal{D}\mathrm{ia}^{\mathrm{lsep.lspat}}, \mathcal{V}\mathrm{stk}^{\mathrm{dct}} \} $) compatibly with the symmetric monoidal structures and with $\ast$-pullbacks. This gives the result.
\end{proof}

\begin{prop}\label{compatibletwo} Let $g:X\to Y$ be a map of decent v-stacks which is representable in locally spatial diamonds and compactifiable of locally finite dim.trg, so in particular $g$ is fine. Then the functor $g_!$ obtained from (2) coincides with the functor $Rg_!$ constructed in \cite[Section 22]{ECoD}.
\end{prop}
Taken together, Propositions \ref{compatibleone} and \ref{compatibletwo} complete the proof of Theorem \ref{mainthmtechnical}.
\begin{proof}Let $A \in \mathcal{D}_{\mathrm{\acute{e}t}}(X,\Lambda)$ be some object. Let $Y_\bullet \to Y$ be a v-hypercover by small coproducts of separated spatial diamonds, so we get a pullback square
\[
\xymatrix{X_\bullet\ar[r]^{g_\bullet}\ar[d]^{v_\bullet} & Y_\bullet\ar[d]^{u_\bullet}\\
X\ar[r]^{g} & Y
}
\]
of small v-stacks. There is a natural equivalence $\mathcal{D}_{\mathrm{\acute{e}t}}(Y,\Lambda) \simeq \mathcal{D}_{\mathrm{\acute{e}t,cart}}(Y_\bullet,\Lambda)$, identifying $g_!A$ with the coCartesian section $n \in \boldsymbol{\Delta} \mapsto g_{n!}v_{n}^{\ast} A \in \mathcal{D}_{\mathrm{\acute{e}t}}(Y_n,\Lambda)$, functorially in $A$.  This identification is an immediate consequence of the proper base change isomorphism $u_n^{\ast}g_! \simeq g_{n!}v_{n}^{\ast}$. Comparing this with the discussion on p. 133-134 of \cite{ECoD}, we are reduced to showing that $g_{n!} \simeq Rg_{n!}$ for all $n \geq 0$. In other words, after resetting the notation, we've reduced the general case of the proposition to the special case where $Y$ is a small coproduct of separated spatial diamonds. 

We immediately reduce further to the case where $Y$ is a separated spatial diamond, so $X$ is a separated and quasiseparated locally spatial diamond and $g:X \to Y$ is compactifiable of locally finite dim.trg. In particular, $g$ is a fine morphism in the category $\mathcal{D}\mathrm{ia}^{\mathrm{qs.lsep.lspat}}$, so the relevant $g_!$ functor is already obtained from Output I of our first iteration of DESCENT. If moreover $g$ is quasicompact, then $g$ is a fine morphism of separated spatial diamonds, so we are actually in the setting of Input I of our first iteration, in which case we already know that $g_! \simeq Rg_!$ by Theorem \ref{seed}. 

Otherwise, let $\mathcal{U}$ be the (filtered) collection of all quasicompact open subdiamonds $U \subset X$; note that each $U$ is automatically separated and spatial. For each $U \in \mathcal{U}$, let $j_U: U \to X$ be the evident map, and set $g_U = g \circ j_U$, so in particular $g_U$ is a fine morphism of separated spatial diamonds. Now, on one hand it is completely formal to see that
\begin{align*}
g_! A \simeq &\, g_! \mathrm{colim}_{U\in\mathcal{U}}j_{U!} j_{U}^{\ast} A\\
\simeq &\, \mathrm{colim}_{U\in\mathcal{U}}g_! j_{U!} j_{U}^{\ast} A\\
\simeq &\, \mathrm{colim}_{U\in\mathcal{U}}g_{U!} j_{U}^{\ast} A.
\end{align*}
Here the first and third lines are trivial, and the second line follows from the fact that $g_!$ commutes with all colimits. On the other hand, in \cite{ECoD} $Rg_! A$ is \emph{defined} as $\mathrm{colim}_{U\in\mathcal{U}} Rg_{U!} j_{U}^{\ast} A$, cf. \cite[Definition 22.13]{ECoD} and the discussion immediately afterwards. By the discussion in the previous paragraph, we already know that $g_{U!} \simeq Rg_{U!}$ for all $U \in \mathcal{U}$, so comparing these expressions, we conclude that $g_! \simeq Rg_!$ as desired.
\end{proof}

\subsection{Cohomological smoothness for stacky maps}
Having completed the proof of Theorem \ref{mainthm}, we now switch notation and write $Rg_!$ for the functor constructed therein, and $Rg^!$ for its right adjoint, in agreement with the notation in \cite{FS, ECoD}.

Recall that in Definition \ref{fine}.ii we have defined $\ell$-cohomologically smooth morphisms between decent v-stacks, and that our definition agrees with \cite[Definition IV.1.11]{FS} by the argument in Remark \ref{fineremark}. Both of these definitions are extrinsic, formulated in terms of the existence of charts with various properties. However, we can now give a purely intrinsic definition, parallel to the definition in the 0-truncated case.

\begin{prop}Let $f:X\to Y$ be a fine map of decent v-stacks. Then the following conditions are equivalent.
\begin{enumerate}

\item The map $f$ is $\ell$-cohomologically smooth in the sense of Definition \ref{fine}.ii.

\item The complex $Rf^! \mathbf{F}_{\ell}$ is invertible, the natural map $Rf^! \mathbf{F}_{\ell} \otimes f^{\ast}(-) \to Rf^!(-)$ is an isomorphism, and these statements hold after any decent base change on $Y$.

\item The complex $Rf^! \mathbf{F}_{\ell}$ is invertible and its formation commutes with any decent base change on $Y$.
\end{enumerate}
\end{prop}

\begin{proof} For (1)$\implies$(2), assume we have charts $a\from V\to Y$, $b\from W\to X$, and an $\ell$-cohomologically smooth morphism $g\from W\to V$ as in the commutative diagram \eqref{EqCommDiagFine}(ii).   Using the $\ell$-cohomological smoothness of $a$, $b$, and $g$, we have $b^*Rf^!\mathbf{F}_{\ell}\isom Rb^!\mathbf{F}_{\ell}^{-1}\otimes Rb^!Rf^!\mathbf{F}_{\ell}\isom Rb^!\mathbf{F}_{\ell}^{-1}\otimes Rg^!Ra^!\mathbf{F}_{\ell}$ is invertible, which implies invertibility of $Rf^!\mathbf{F}_{\ell}$ since invertibility can be detected v-locally.   For the second part of the claim, suppose $A\in D_{\et}(Y,\mathbf{F}_{\ell})$.  It is enough to check that $Rf^!\mathbf{F}_\ell\otimes f^*A\to Rf^!A$ is an isomorphism after pullback through $b$,  and one sees after twisting by $Rb^!\mathbf{F}_\ell$ that this is equivalent to the statement that $Rg^!Ra^!\mathbf{F}_\ell\otimes g^*a^*A\to Rg^!Ra^!A$ is an isomorphism, which is true because $a\circ g$ is $\ell$-cohomologically smooth.   Since the property (1) is stable under decent base change on $Y$, so are the properties we have just verified.

Now we prove (2)$\implies$(1):  Let $a\from V\to Y$ be a chart, and let $b\from W\to X$ be the base change of $a$ through $f$,  so that once again we have a diagram as in \eqref{EqCommDiagFine}.  By hypothesis, $Rg^!\Lambda$ is invertible, and $Rg^!\mathbf{F}_\ell\otimes g^*(-)\to Rg^!(-)$ is an isomorphism.   Let $h\from U\to W$ be a chart.  Applying $h^!\isom h^!\mathbf{F}_\ell\otimes h^*$, we find that $R(g\circ h)^!\mathbf{F}_\ell\otimes (g\circ h)^*(-)\to R(g\circ h)^!(-)$ is an isomorphism.   This means that $g\circ h$ is an $\ell$-cohomologically smooth morphism between charts for $X$ and $Y$ respectively,  and so $f$ is $\ell$-cohomologically smooth in the sense of Definition \ref{fine}.ii.  

By the discussion immediately following \cite[Proposition IV.2.33]{FS}, a compactifiable morphism $g$ between locally spatial diamonds with finite dim. trg is $\ell$-cohomologically smooth (in the sense of \cite{ECoD}) if and only if $Rg^!\mathbf{F}_\ell$ is invertible and its formation commutes with any base change.    The above arguments can now be adapted to give the equivalence between (1) and (3).

\end{proof}

\bibliographystyle{amsalpha}
\bibliography{enhancednew}

\end{document}